\documentclass[11pt]{article}
\usepackage[utf8]{inputenc}
\usepackage[english]{babel}
\usepackage[a4paper, margin=0.9in]{geometry}
\usepackage{amsfonts,amssymb,amsmath,amsthm,hyperref,graphicx,enumerate,wasysym,xcolor,fancyvrb}
\usepackage{titlesec}
\usepackage{mathtools,mathrsfs}
\mathtoolsset{showonlyrefs}
\usepackage[nottoc,numbib]{tocbibind}
\usepackage{chngcntr}
\usepackage[safeinputenc,maxbibnames=99]{biblatex}
\usepackage{csquotes}

\title{Heteroclinic traveling waves of 1D parabolic systems with degenerate stable states}
\author{Ramon Oliver-Bonafoux\thanks{Sorbonne Université, Laboratoire Jacques-Louis Lions. 4 Place Jussieu, 75005 Paris (France). \linebreak Email: \texttt{ramon.oliver\_bonafoux@upmc.fr}}}

\numberwithin{equation}{section}

\def\R{{\mathbb R}}

\def\N{{\mathbb N}}

\def\CA{\mathcal{A}}

\def\CD{\mathcal{D}}
\def\CS{\mathcal{S}}

\def\CC{\mathcal{C}}

\def\CE{\mathcal{E}}

\def\FU{\mathfrak{U}}

\def\FP{\mathfrak{P}}
\def\FM{\mathfrak{M}}

\def\scrH{\mathscr{H}}
\def\scrL{\mathscr{L}}

\def\scrF{\mathscr{F}}

\def\sfC{\mathsf{C}}
\def\sfT{\mathsf{T}}

\def\Fu{\mathfrak{u}}

\def\Fm{\mathfrak{m}}

\def\Fw{\mathfrak{w}}
\def\Fq{\mathfrak{q}}

\def\Fh{\mathfrak{h}}

\def\bV{\mathbf{V}}

\def\bU{\mathbf{U}}
\def\bE{\mathbf{E}}
\def\bW{\mathbf{W}}

\def\be{\mathbf{e}}

\def \bb{\mathbf{b}}

\def\bv{\mathbf{v}}
\def\bu{\mathbf{u}}

\def\sstar{{\star\star}}

\def\eps{\varepsilon}
\def\loc{\mathrm{loc}}

\def\dist{\mathrm{dist}}

\titleclass{\subsubsubsection}{straight}[\subsection]

\newcounter{subsubsubsection}[subsubsection]
\renewcommand\thesubsubsubsection{\thesubsubsection.\arabic{subsubsubsection}}
\renewcommand\theparagraph{\thesubsubsubsection.\arabic{paragraph}} 

\titleformat{\subsubsubsection}
  {\normalfont\normalsize\bfseries}{\thesubsubsubsection}{1em}{}
\titlespacing*{\subsubsubsection}
{0pt}{3.25ex plus 1ex minus .2ex}{1.5ex plus .2ex}

\makeatletter
\renewcommand\paragraph{\@startsection{paragraph}{5}{\z@}%
  {3.25ex \@plus1ex \@minus.2ex}%
  {-1em}%
  {\normalfont\normalsize\bfseries}}
\renewcommand\subparagraph{\@startsection{subparagraph}{6}{\parindent}%
  {3.25ex \@plus1ex \@minus .2ex}%
  {-1em}%
  {\normalfont\normalsize\bfseries}}
\def\toclevel@subsubsubsection{4}
\def\toclevel@paragraph{5}
\def\toclevel@paragraph{6}
\def\l@subsubsubsection{\@dottedtocline{4}{7em}{4em}}
\def\l@paragraph{\@dottedtocline{5}{10em}{5em}}
\def\l@subparagraph{\@dottedtocline{6}{14em}{6em}}
\makeatother

\newtheorem{theorem}{Theorem}
\newtheorem{definition}{Definition}
\newtheorem{proposition}{Proposition}
\newtheorem{lemma}{Lemma}
\newtheorem{corollary}{Corollary}

\theoremstyle{definition}
\newtheorem{remark}{Remark}

\newtheorem{asu}{}

\newtheorem{hyp}{}

\numberwithin{proposition}{section}
\numberwithin{lemma}{section}
\numberwithin{remark}{section}
\numberwithin{definition}{section}
\numberwithin{example}{section}
\numberwithin{corollary}{section}
\numberwithin{figure}{section}

\begin{document}
\maketitle

\begin{abstract}
We study the existence of traveling waves for the parabolic system
\begin{equation}
\partial_t w - \partial_{x}^2 w = -\nabla_{\bu} W(w) \mbox{ in } [0,+\infty) \times \R
\end{equation}
where $W$ is a potential bounded below and possessing two minima at \textit{different levels}. We say that $\Fw$ is a traveling wave solution of the previous equation if there exist $c^\star>0$ and $\Fu \in \CC^2(\R,\R^k)$ such that $\Fw(t,x)=\Fu(x-ct)$. For a class of potentials $W$, heteroclinic traveling waves of the previous equation where shown to exist by Alikakos and Katzourakis \cite{alikakos-katzourakis}. More precisely, assuming the existence of two local minimizers of $W$ at \textit{different} levels which, in addition, satisfy some non-degeneracy assumptions, the authors in \cite{alikakos-katzourakis} show the existence of a speed $c^\star>0$ and profile $\Fu \in \CC^2(\R,\R^k)$ such that $\Fu$ connects the two local minimizers at infinity.  In this paper, we show that the non-degeneracy assumption on the local minima can be dropped and replaced by another one which allows for potentials possessing degenerate minima. As we do in \cite{oliver-bonafoux-tw}, our main result is in fact proven for curves which take values in a general Hilbert space and the main result is deduced as a particular case, in the spirit of the earlier works by Monteil and Santambrogio \cite{monteil-santambrogio} and Smyrnelis \cite{smyrnelis} devoted to the existence of stationary heteroclinics.
\end{abstract}
\section{Introduction}
We are concerned with the following one-dimensional parabolic system of reaction-diffusion type
\begin{equation}\label{AC_unbalanced}
\partial_t w - \partial_{x}^2 w = -\nabla_{\bu} W(w) \mbox{ in } [0,+\infty) \times \R,
\end{equation}
where $w:[0,+\infty) \times \R \to \R^k$ and $W$ is an \textit{unbalanced} double-well potential with possibly degenerate minima, meaning that we can find two minimum points of $W$ at different levels and we do not assume that $D^2W$ is positive definite at such points. For potentials $W$ which possess two non-degenerate minimum points $a^-, a^+$ in $\R^k$ such that $W(a^-)<0=W(a^+)$ satisfying some additional assumptions (namely, non-degeneracy and local radial monotonicity), Alikakos and Katzourakis \cite{alikakos-katzourakis} (see also Lucia, Muratov and Novaga \cite{lucia-muratov-novaga} and Risler  \cite{risler}) showed the existence of traveling wave solutions for \eqref{AC_unbalanced} with heteroclinic behavior at infinity. More precisely, they showed that there exist $c^\star>0$ and $\Fu \in \CC^2(\R,\R^k)$ such that
\begin{equation}\label{PROFILE_eq}
-c^\star \Fu' - \Fu''=-\nabla_{\bu} W(\Fu) \mbox{ in } \R
\end{equation}
and
\begin{equation}\label{PROFILE}
\lim_{t \to \pm \infty} \Fu(t)=a^\pm.
\end{equation}
Such a thing implies that $\Fw:[0,+\infty) \times \R \to \R^k$ defined as
\begin{equation}
\Fw(t,x)=\Fu(x-ct)
\end{equation}
for $(t,x) \in [0,+\infty) \times \R$ is a solution to \eqref{AC_unbalanced}. More precisely, $\Fw$ is a traveling wave solution of \eqref{AC_unbalanced} which propagates with speed $c^\star$ and with profile $\Fu$. Our main contribution here is to show that such solutions exist when we consider some class potentials $W$ with possibly degenerate and non-isolated minima. More precisely, we assume that $W$ takes the minimum values on sets, instead of isolated points, and that it satisfies suitable properties around such sets. Our motivation comes from a closely related work by the author \cite{oliver-bonafoux-tw}. There, traveling waves solutions for some classes of two dimensional parabolic Allen-Cahn systems are obtained. The approach of the proof is to deduce the result from a more general one, proved in an abstract setting on Hilbert spaces, following ideas that worked for similar problems (see Monteil and Santambrogio \cite{monteil-santambrogio} and Smyrnelis \cite{smyrnelis}). More precisely, one assumes that the potential $W$ is defined in a Hilbert space $\scrH$ which has no restriction in its  dimension, so that in case $\scrH$ is infinite-dimensional and $W$ is suitably chosen, one recovers a system which has more than one dimension in space.  In this paper, we show that the abstract approach used in our previous work can be also used to recover results in the \textit{finite-dimensional setting}. For this we mean that $\scrH=\R^k$, so that the resulting equation is of the type \eqref{AC_unbalanced}, which is 1D on space. These results are not included in \cite{alikakos-katzourakis}. The abstract result of \cite{oliver-bonafoux-tw} applies to our problem, but the main result of this paper follows by a modified version of such abstract result, in which a key assumption in \cite{oliver-bonafoux-tw} is replaced by another one playing the same role. Such an assumption is the natural adaptation of the one used by Alikakos and Katzourakis \cite{alikakos-katzourakis} to the case of degenerate minima, and we use it following the ideas from \cite{alikakos-katzourakis} but with more involved arguments, which is not surprising since the fact that the minima might be degenerate adds extra difficulties. Nevertheless, at present we are only able to apply the abstract result of this paper to the finite-dimensional problem, meaning that we do not use the full strength of the abstract setting, which could be useful in other situations. We also mention that the result we obtain is weaker than that by Alikakos and Katzourakis \cite{alikakos-katzourakis} since we cannot prove convergence of the profile to the global minimum at $-\infty$, only that it stays close to it for small enough times. This might obey to merely technical limitations, but it could also be due to the fact that we allow for degenerate minima. However, we show that under an additional assumption (essentially, either an upper bound $W(a^+)$ or the set of global minimizers is a singleton, see \ref{asu_bc}) one essentially recovers the full result from \cite{alikakos-katzourakis} for our setting.

The scheme of the proof of the abstract result of this paper is analogous to that in \cite{oliver-bonafoux-tw}, which is based on the variational approach of Alikakos and Katzourakis \cite{alikakos-katzourakis}, inspired by that in Alikakos and Fusco \cite{alikakos-fusco}. While the existence of a variational structure in the context of (some) reaction-diffusion problems is known since Fife and McLeod \cite{fife-mcleod77,fife-mcleod81}, it has not been widely used in the previous literature. Besides \cite{alikakos-katzourakis} and our \cite{oliver-bonafoux-tw}, other references which use such a variational structure for studying traveling waves in reaction-diffusion problems are Bouhours and Nadin \cite{bouhours-nadin}, Lucia, Muratov and Novaga \cite{lucia-muratov-novaga04,lucia-muratov-novaga}, Muratov \cite{muratov}, Risler \cite{risler,risler2021-1,risler2021-2} and, more recently, Chen, Chien and Huang \cite{chen-chien-huang}. 

In \cite{alikakos-katzourakis,oliver-bonafoux-tw}, one considers a family of constrained minimization problems and the main difficulty of the problem can be reduced to excluding a degenerate oscillatory behavior for the minimizers of these problems, as the energy density under consideration changes its sign due to the fact that the connected minima are at different levels. In \cite{alikakos-katzourakis}, this is done by imposing a non-degeneracy assumption, as well as radial monotonicity properties, on the local minima and then using the ODE system along with optimality of the minimizers. For several reasons, an assumption of this type was not available to us in \cite{oliver-bonafoux-tw}, so that we needed to replace it by a different one which allowed an analogous type of conclusion. Such an assumption consists essentially on an upper bound on the difference between the energy of the minima, which implies that the energy of the constrained minimizers  to. The abstract result of this paper further replaces the previous assumption by another one which plays the same role (but which does not apply to the problem considered in \cite{oliver-bonafoux}) and still allows for potentials with degenerate minima. More precisely, it allows us to apply to the constrained minimizers an argument based on the use of the ODE system, in the spirit of \cite{alikakos-katzourakis} but more involved in some aspects. The reader is referred to \cite{oliver-bonafoux-tw} for a detailed scheme of proof as well as a discussion on related literature.

In order to conclude this paragraph, we recall that in the realm of \textit{scalar} reaction-diffusion problems one can obtain important results of existence and qualitative properties of traveling waves by means of the maximum principle and comparison results. There is a large literature on the subject, the more classical papers are Fife and McLeod \cite{fife-mcleod77,fife-mcleod81},Aronson and Weinberger \cite{aronson-weinberger}, Berestycki and Nirenberg \cite{berestycki-nirenberg}, all devoted to the Fisher KPP equation (introduced by Fisher \cite{fisher}, Kolmogorov, Petrovsky and Piskunov \cite{kpp}). In the context of the Allen-Cahn equation (which falls in a different framework than the Fisher-KPP), the stability of traveling waves was shown by Matano, Nara and Taniguchi \cite{matano-nara-taniguchi}. However, the maximum and comparison principles are no longer available (in general) for vector-values problems (i. e., systems), so that one needs different tools. There is also a huge amount of research in this direction, we refer to the books by Smoller \cite{smoller} as well as Volpert, Volpert and Volpert \cite{volpertx3}. However, the use of variational methods is not covered in the previously cited books and it has not been extensively used in the parabolic context, which is in contrast with the framework of dispersive equations.

\subsection*{Acknowledgments}

I wish to thank my PhD advisor Fabrice Bethuel for his encouragement and many useful comments and remarks during the elaboration of this paper.

\includegraphics[scale=0.3]{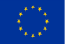}
This program has received funding from the European Union's Horizon 2020 research and innovation programme under the Marie Skłodowska-Curie grant agreement No 754362.
\tableofcontents
\tableofcontents
\section{Statement of the main results}
\subsection{Main assumptions and result}
The main result of this paper is Theorem \ref{THEOREM_W_main}, which establishes the existence of an heteroclinic traveling wave with the uniqueness properties on the speed already found in \cite{alikakos-katzourakis} and exponential convergence at $+\infty$. The behavior at $-\infty$ is however weaker than that proved in \cite{alikakos-katzourakis}, which might be due to the fact that we allow for degenerate minima. The necessary assumptions for Theorem \ref{THEOREM_W_main} are \ref{asu_unbalancedW}, \ref{asu_projection_inverse} and \ref{asu_conv_mon}.  In Theorem \ref{THEOREM_W_bc} we show that one can obtain convergence of the profile at $-\infty$ (so that one recovers the full result from \cite{alikakos-katzourakis} if one assumes either an upper bound on the difference of the energy of the minima or that the set of global minimizers is a singleton, see \ref{asu_bc}. Finally, in Theorem \ref{THEOREM_W_speed} we show that the formula and min-max characterization of the speed given in \cite{alikakos-katzourakis} also holds here, again under assumption \ref{asu_bc}.
\begin{asu}\label{asu_unbalancedW}
$W \in \CC^3(\R^k)$. There exist $\CA^-$ and $\CA^+$ subsets of $\R^k$ and $\Fh <0$ such that for all $a^- \in \CA^-$ we have $W(a^-)=\Fh<0$, for all $a^+ \in \CA^+$ we have $W(a^+)=0$ and for all $u \in \R^k$ we have $W(\bu) \geq \Fh$. There exist $\rho_0^\pm>0$ such that for all $\bu \in \R^k$ such that $\dist(\bu,\CA^\pm) \leq \rho_0^\pm$ it holds
\begin{equation}\label{coercivity_W}
\dist(\bu,\CA^\pm)^2 \leq C^\pm (W(\bu)-\min\{ \pm(-\Fh),0\})
\end{equation}
for some $C^\pm>0$. Moreover, there exists a unique $a^\pm(\bu) \in \CA^\pm$ such that
\begin{equation}
\dist(\bu,\CA^\pm)=\lvert \bu-a^\pm(\bu) \rvert
\end{equation}
and the mappings $p^\pm : \bu \to a^\pm(\bu)$ defined on
\begin{equation}
\CA^\pm_{\rho^\pm}:= \{ \bu \in \R^k: \dist(\bu,\CA^\pm) \leq r_0^\pm \}
\end{equation}
are $C^2$.  Moreover, we have that there exist $c_0>0$ and $R_0>0$ such that for all $\bu \in \R^k$ such that $\lvert \bu \rvert \geq R_0$, we have
\begin{equation}\label{W_R0}
\langle \nabla W(\bu), \bu \rangle \geq c_0 \lvert \bu \rvert^2
\end{equation}
\end{asu}
In particular, the existence of the projection mappings $p^\pm$ holds if  the sets $\CA^\pm$ are convex and smooth. The next assumption writes as follows:
\begin{asu}\label{asu_projection_inverse}
One of the two follows:
\begin{enumerate}
\item $\CA^\pm$ is bounded.
\item For any $(\bu,a^\pm) \in \CA_{\rho_0^\pm}^\pm \times \CA^\pm$, there exist maps $p^\pm_{(\bu,a^\pm)}: \R^k \to \R^k$ such that $p^\pm_{(\bu,a^\pm)}(\bu)=a^\pm$, $W(p^\pm_{(\bu,a^\pm)}(\bu))=W(\bu)$ and $\dist(p^\pm_{(\bu,a^\pm)}(\bu),\CA^\pm)=\dist(\bu,\CA^\pm)$. Moreover, $p^\pm_{(\bu,a^\pm)}$ is differentiable and $\lvert D(p^\pm_{(\bu,a^\pm)})(\bu_1,\bu_2) \rvert = \lvert \bu_2 \rvert$ for any $(\bu_1,\bu_2) \in \left( \R^k\right)^2$.
\end{enumerate}
\end{asu}
In particular, if $W$ is invariant with respect to some group action inside $\CA_{\rho_0^\pm}^\pm$ and $\CA^\pm$ is a single orbit with respect to such an action, then 2. in \ref{asu_projection_inverse} is met. However, less rigid structures are also allowed by 2. Finally, for each $h \in \R$ define the level set
\begin{equation}
W^h:= \{ \bu \in \R^k: W(\bu) \leq h\}
\end{equation}
and notice that by continuity and \ref{asu_unbalancedW} we can find $h_0>0$ such that
\begin{equation}
W^{h_0}=W^{h_0,-} \cup W^{h_0,+}
\end{equation}
with $W^{h_0,\pm}$ closed and disjoint and such that $W^{h_0,+} \subset \CA_{\rho_0^+/2}^+$. In particular
\begin{equation}
W^{h_0,-} \cap \CA_{\rho_0^+/2}^+ = \emptyset.
\end{equation}
For each $h \leq h_0$, we can then write
\begin{equation}
W^{h}=W^{h,-} \cup W^{h,+}
\end{equation}
with $W^{h,\pm}$ closed and disjoint and such that $W^{h,\pm} \subset W^{h_0,\pm}$. It is then clear that
\begin{equation}
W^{0}=W^{0,-} \cup \CA^+, W^{\Fh}=\CA^-
\end{equation}
and for all $h < 0$,
\begin{equation}
W^{h}=W^{h,-}.
\end{equation}
Under these notations, our last assumption writes as follows:
\begin{asu}\label{asu_conv_mon}
For any $h \leq h_0$, the set $W^{h,-}$ is convex. For any $h \leq h_0$, $\bu \in W^{h,-}$ and $a^- \in \CA^-$ define the set
\begin{equation}
I(h,\bu,a^-):= \{ \lambda \in \R: W(a^-+\lambda(\bu-a^-)) > h \mbox{ and } a^-+\lambda(\bu-a^-)\in W^{h_0,-}\}.
\end{equation}
Then, there exists $h_- \in (\Fh,0)$ such that
\begin{enumerate}
\item $W^{h_-}\subset \CA^-_{\rho_0^-/2}$.
\item For some constant $\sigma>0$ it holds that for $\bu \in W^{h_0,-}$, $a^- \in \CA^-$ and $\theta \in I((\Fh+h_-)/2,\bu,a^-)$ we have
\begin{equation}
\frac{d}{d\lambda}(W(a^-+\lambda(\bu-a^-)))(\theta) \geq \sigma.
\end{equation}
\item For all $h \in (\Fh,h_-]$, there exists $\sigma(h)>0$ such that for all $\bu \in W^{h_-}$ such that $W(\bu) \geq h$, there exists $\delta(\bu)>0$ such that for all $\theta \in (1-\delta(\bu),1+\delta(\bu))$ we have 
\begin{equation}\label{sigma_h}
\frac{d}{d\lambda}(W(p^-(\bu)+\lambda(\bu-p^-(\bu)))(\theta) \geq \sigma(h)
\end{equation}
with $p(\bu)$ the projection given in \ref{asu_unbalancedW}.
\end{enumerate}
\end{asu}
Essentially, assumption \ref{asu_conv_mon} imposes some convexity on some suitable subsets of the level sets of $W$, as well as some uniform monotinicity on segments. Some explanatory designs are shown in Figures \ref{fig_item1}, \ref{fig_item2}, \ref{fig_item3}. One can see assumption \ref{asu_projection_inverse} as an adaptation of the key assumption by Alikakos and Katzourakis to the case of degenerate minima. Indeed, in case $\CA^-$ is reduced to a singleton, \ref{asu_conv_mon} essentially reduces to the hypothesis formulated in \cite{alikakos-katzourakis}. Finally, recall that, as in \cite{alikakos-katzourakis}, multi-well potentials can satisfy the assumptions \ref{asu_unbalancedW}, \ref{asu_projection_inverse} and \ref{asu_conv_mon} as long as they possess a local minimum at a level higher than 0 and they are modified by an additive constant. We now come back to the following equation for a pair $(c,u)$
\begin{equation}\label{PROFILE_eq_Thm}
-c u' - u''=-\nabla_{\bu} W(u) \mbox{ in } \R
\end{equation}
and, assuming that \ref{asu_conv_mon} holds, consider the conditions at infinity
\begin{equation}\label{conditions_infinity}
\exists T^- \in \R: \forall t \leq T^-, \hspace{2mm} u(t) \in W^{h_-} \mbox{ and } \exists a^+(u) \in \CA^+: \lim_{t \to +\infty} \lvert u(t)-a^+(u) \rvert =0.
\end{equation}
For $c>0$, we consider as in \cite{alikakos-katzourakis}, the following weighted functional introduced in Fife and McLeod \cite{fife-mcleod77,fife-mcleod81}
\begin{equation}
E_c(v):= \int_\R e_c(v(t))dt:= \int_\R \left[ \frac{\lvert v'(t) \rvert^2}{2}+W(v(t)) \right] e^{ct}dt
\end{equation}
where $v$ belongs to the class
\begin{equation}
\CS:= \left\{ v \in H^1_\loc(\R,\R^k): \exists T \geq 1: \forall t \geq T, \hspace{2mm} v(t) \in \CA^+_{\rho_0^+/2} \mbox{ and } \forall t \leq -T, \hspace{2mm} v(t) \in \CA^-_{\rho_0^-/2} \right\}
\end{equation}
so that \eqref{PROFILE_eq_Thm} is (at least formally) the Euler-Lagrange equation of $E_c$. As in \cite{alikakos-katzourakis}, we shall find the solution profile $\Fu$ as a critical point (in fact, a global minimizer) of $E_{c^\star}$ in $\CS$ for a suitable $c^\star$ which, in addition, satisfies some uniqueness properties. In fact, assuming that \ref{asu_conv_mon} holds we shall consider the class
\begin{equation}
\overline{\CS}:= \left\{ v \in H^1_\loc(\R,\R^k): \exists T \geq 1: \forall t \geq T, \hspace{2mm} v(t) \in \CA^+_{\rho_0^+/2} \mbox{ and } \forall t \leq -T, \hspace{2mm} v(t) \in W^{h_-} \right\},
\end{equation}
which by 1. in \ref{asu_conv_mon} satisfies $\overline{S} \subset \CS$. As it was pointed out in \cite{alikakos-katzourakis}, for any $c>0$, $v \in \CS$ and $\tau \in \R$ we have
\begin{equation}
E_c(v(\cdot+\tau))=e^{-c\tau}E_c(v)
\end{equation}
which implies that for all $c>0$
\begin{equation}\label{inf_v}
\inf_{v \in \CS}E_c(v), \inf_{v \in \overline{\CS}}E_c(v) \in \{-\infty,0\}
\end{equation}
and that in case the infimum above is 0, then for any $v \in \CS$ with $E_c(v)>0$ we have that $E_c(v(\cdot+n)) \to 0$ as $n \to +\infty$. This remark shows that one cannot expect to solve the minimization problem \eqref{inf_v} directly, so that an indirect approach is needed. We can now state the main result of this paper:
\begin{theorem}[Main Theorem]\label{THEOREM_W_main}
Assume that \ref{asu_unbalancedW}, \ref{asu_projection_inverse} and \ref{asu_conv_mon} hold. Then, we have that
\begin{enumerate}
\item \textbf{Existence}. There exist $c^\star>0$ and $\Fu \in \CC^2_\loc(\R,\R^k) \cap \overline{S}$ such that $(c^\star,\Fu)$ fulfills \eqref{PROFILE_eq_Thm} and \eqref{conditions_infinity} as well as the variational characterization
\begin{equation}
E_{c^\star}(\Fu)=0= \inf_{v \in \overline{\CS}}E_{c^\star}(v).
\end{equation}
\item \textbf{Uniqueness of the speed}. Assume that $\overline{c^\star}>0$ is such that
\begin{equation}
\inf_{v \in \overline{\CS}}E_{\overline{c^\star}}(v)=0
\end{equation}
and that $\overline{\Fu} \in \CS$ is such that $(\overline{c^\star},\overline{\Fu})$ solves \eqref{PROFILE_eq_Thm} and $E_{\overline{c^\star}}(\overline{\Fu})<+\infty$.  Then, $\overline{c^\star}=c^\star$.
\item \textbf{Exponential convergence}. The convergence of $\Fu$ at $+\infty$ is exponential: There exists $\FM^+>0$ such that for all $t \in \R$
\begin{equation}
\lvert \Fu(t)-a^+(\Fu) \rvert \leq \FM^+ e^{-c^\star t}
\end{equation}
where $a^+(\Fu)$ is given by \eqref{conditions_infinity}.
\end{enumerate}
\end{theorem}
\begin{remark}
Notice that the conditions at infinity \eqref{conditions_infinity} show that $\Fu$ is not constant, since we clearly have that $\CA^-_{\rho_0^-/2} \cap \CA^+_{\rho^+_0/2}=\emptyset$. Regarding the uniqueness of the speed, in particular it holds that $c^\star$ is unique among the class of global minimizers.
\end{remark}
\begin{figure}[h!]
\centering
\includegraphics[scale=0.3]{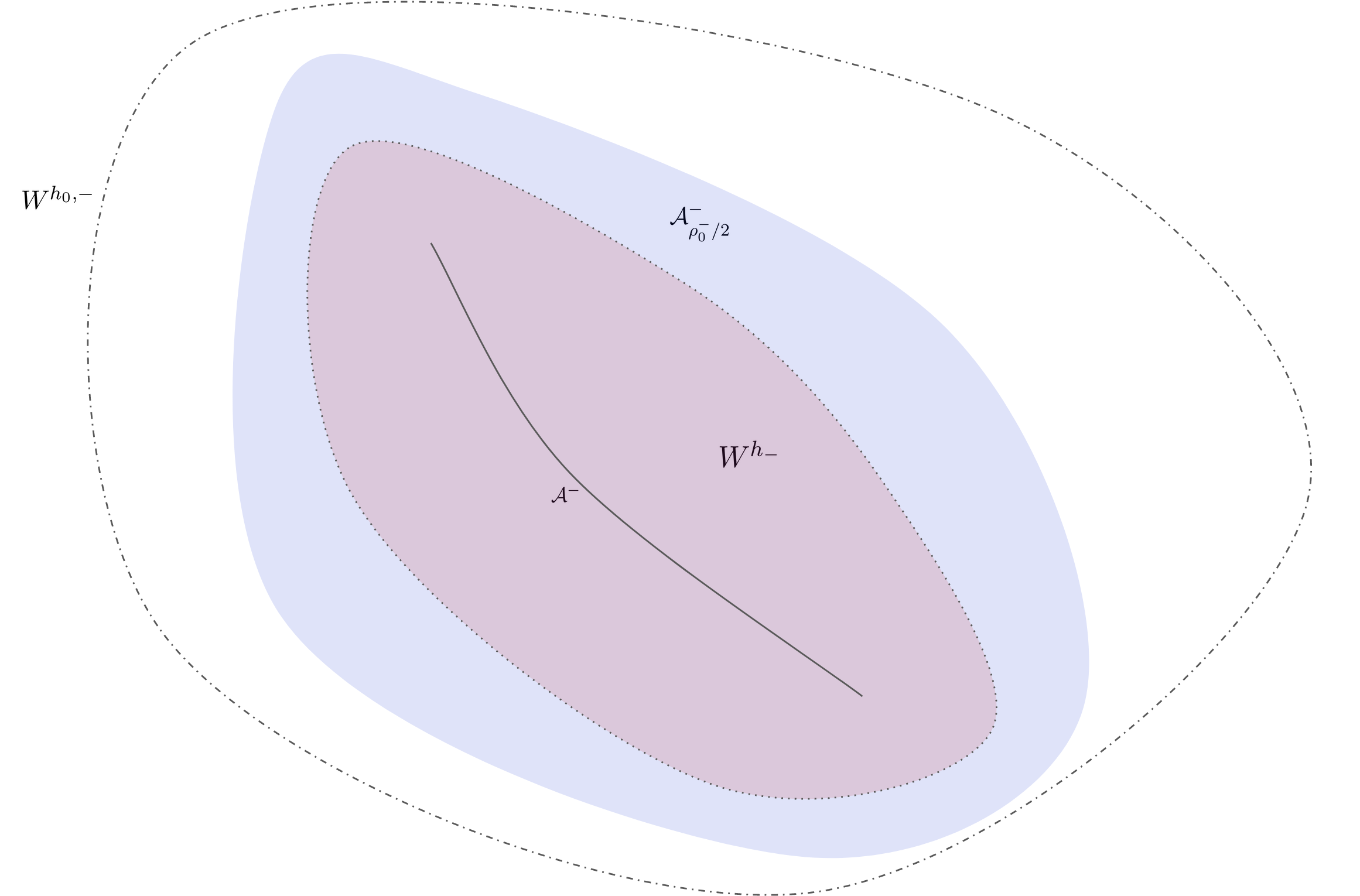}
\caption{Representation of 1. in \ref{asu_conv_mon}. The curve represents
the set $\CA^-$, while the inner convex shadowed region corresponds to the level set $W^{h_-}$,
which is contained in $\CA^-_{\rho_0^-/2}$ , the larger shadowed region. Finally, the outer
punctured line contains the convex region representing $W^{h_0,-}$.}\label{fig_item1}
\end{figure}
\begin{figure}[h!]
\centering
\includegraphics[scale=0.3]{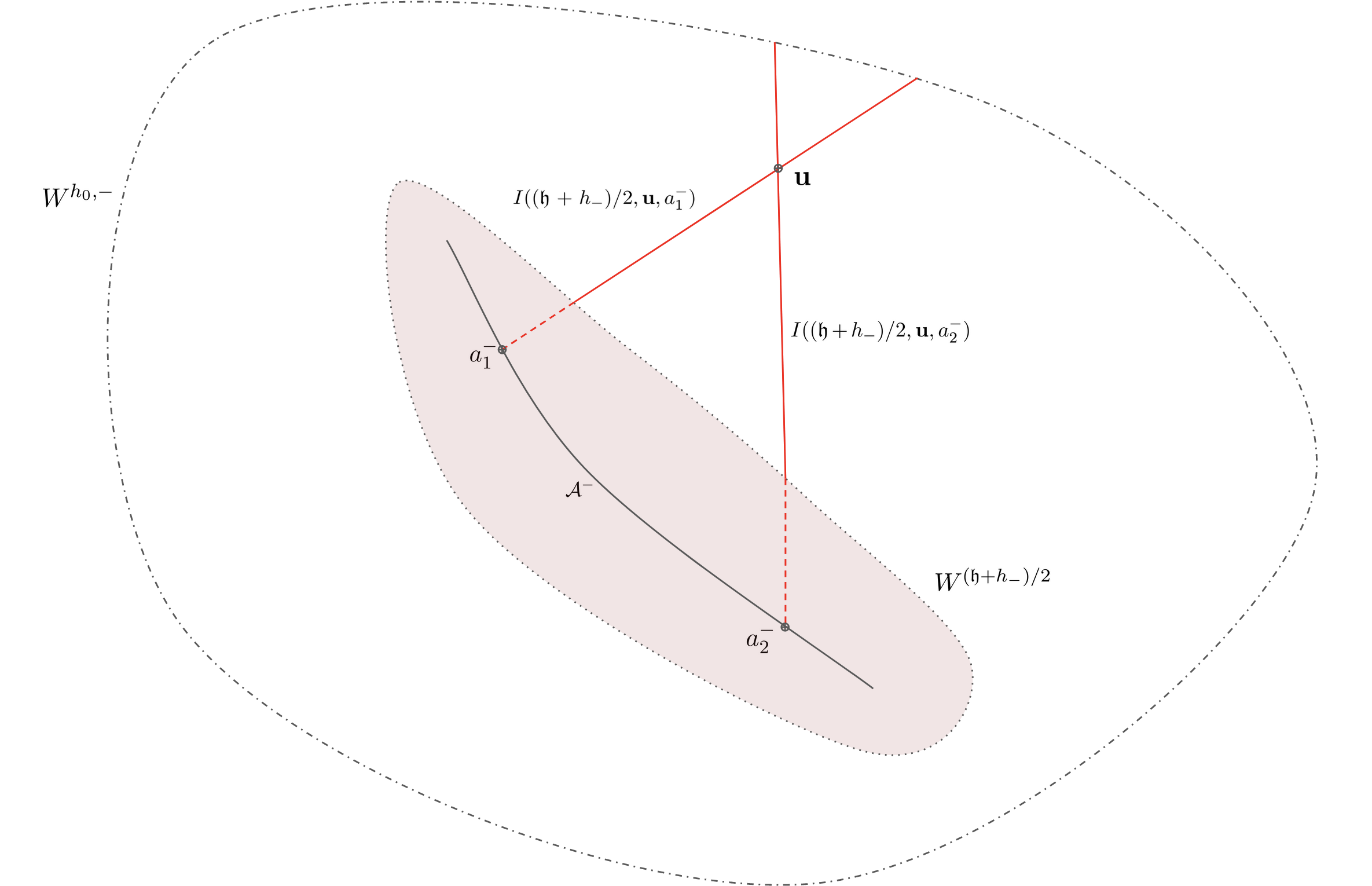}
\caption{Representation of 2. in \ref{asu_conv_mon}. We have $\bu$, which is contained
in $W^{h_0,-}$ but not in $W^{(\Fh+h_-)/2}$ . The full lines represent the segments $I((\Fh+h_-)/2,\bu,a_1^-)$ and $I((\Fh+h_-)/2,\bu,a_2^-)$ for $a_1^-$ and $a_2^-$ in $\CA^-$. The discontinuous lines complete the previous segments into the segment starting in $a^-_1$ and $a^-_2$.}\label{fig_item2}
\end{figure}
\begin{figure}[h!]
\centering
\includegraphics[scale=0.5]{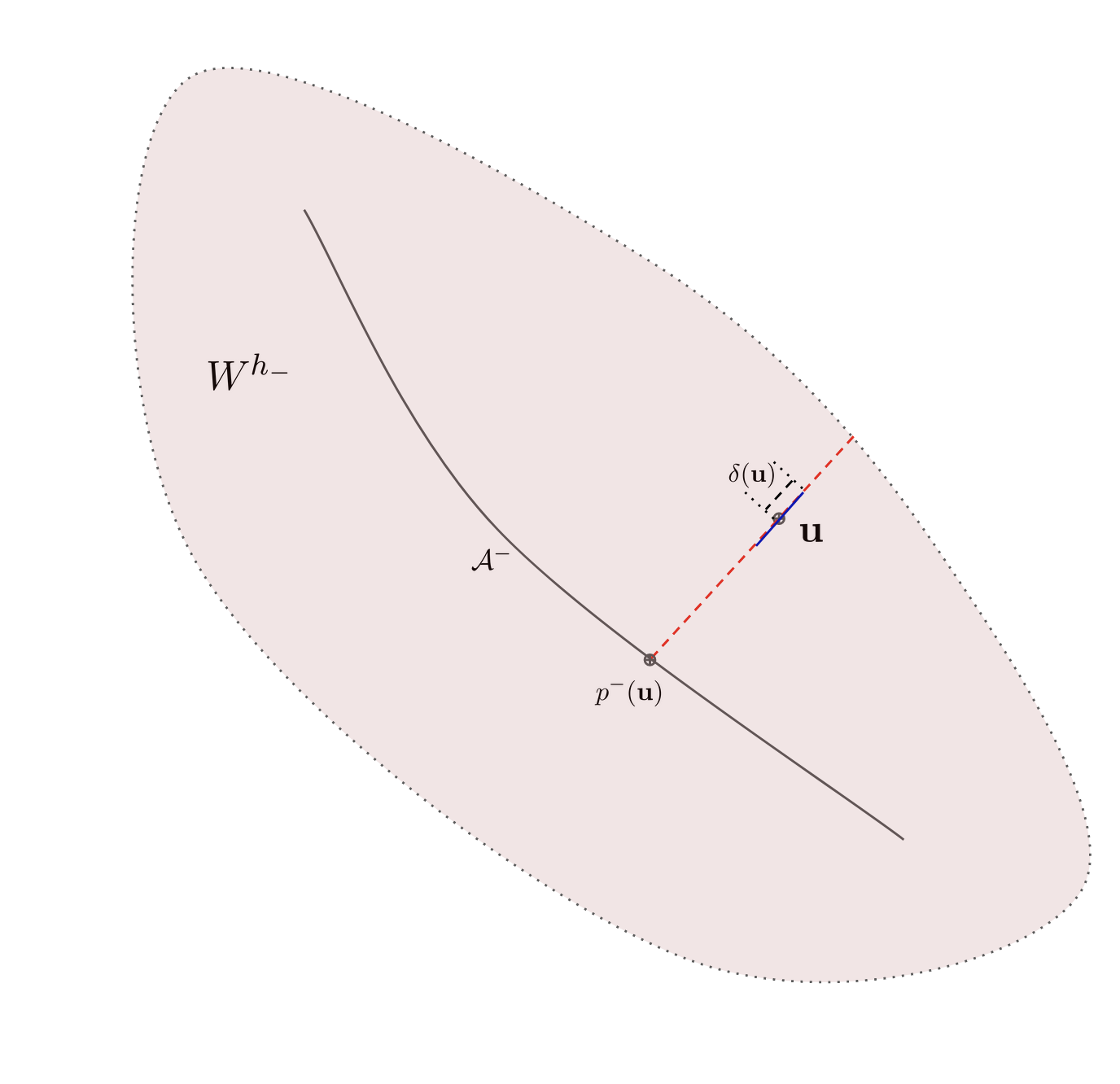}
\caption{Representation of 3. in \ref{asu_conv_mon}. We take $\bu \in W^{h_-}$ not contained in $\CA^-$. The discontinuous red line is a segment starting in $p^-(\bu)$ and containing $\bu$. The blue line represents the small portion of the segment centered in $\bu$ contained in the line going trough $p^-(\bu)$ and with length $2\delta(\bu)$, $\delta(\bu)> 0$. We ask uniform strict monotonicity property \eqref{sigma_h} to hold in this segment.}\label{fig_item3}
\end{figure}
Once the main result has been stated, we give some properties on the solution $(c^\star,\Fu)$, which are essentially the adaptation from \cite{oliver-bonafoux-tw} to the current setting.
\subsection{Conditions at infinity}
We now explain how the condition at infinity \eqref{conditions_infinity} can be upgraded. In particular, we would like conditions at infinity of the type
\begin{equation}\label{boundary_conditions_def}
\exists a^\pm(u) \in \CA^\pm: \lim_{t \to \pm \infty} \lvert u(t)-a^\pm(u) \rvert =0.
\end{equation}
Under the previous assumptions, we have no proof that $(c,\Fu)$ satisfies \eqref{boundary_conditions_def}. However, we can establish such a behavior adding the following assumption:
\begin{asu}\label{asu_bc}
We have one of the two following situations:
\begin{enumerate}
\item $\CA^-:=\{a^-\}$.
\item We have that
\begin{equation}
-\Fh < \frac{\left( d \eta \right)^2}{2}
\end{equation}
where
\begin{equation}\label{d}
d:= \dist_{\R^k}(	\CA^+_{\rho_0^+/2},W^{h_0,-})
\end{equation}
which is positive by the definition of $h_0$ and
\begin{equation}\label{eta}
\eta:= \frac{1}{\frac{1}{2}(C^-)^2((C^-)^2+(C^-+1)^2)+C^-}>0,
\end{equation}
where $C^-$ is the constant from \ref{asu_unbalancedW}.
\end{enumerate}
\end{asu}
Essentially, \ref{asu_bc} requires that either $\CA^-$ is a singleton (while $\CA^+$ does not need to be) or that the value $-\Fh$ is not too large. Then we have the following:
\begin{theorem}\label{THEOREM_W_bc}
Assume that \ref{asu_unbalancedW}, \ref{asu_projection_inverse}, \ref{asu_conv_mon} and \ref{asu_bc} hold. Let $(c^\star,\Fu)$ be a solution given by Theorem \ref{THEOREM_W_main}. Then, $\Fu$ also satisfies \eqref{boundary_conditions_def}. Moreover, the convergence is exponential in the following way: We have that $c^\star<\eta$, $\eta$ as in \eqref{eta}, and there exists $\FM^->0$ such that for all $t \in \R$
\begin{equation}
\lvert \Fu(t)-a^-(\bu) \rvert \leq \FM^- e^{(\eta-c^\star)t}
\end{equation}
\end{theorem}

\subsection{Min-max characterization of the speed}
We now give some results which characterize the speed $c^\star$. As before, such results are very close to the ones obtained in \cite{alikakos-katzourakis} and \cite{oliver-bonafoux-tw}. More specifically, in case \ref{asu_bc} holds, the speed $c^\star$ satisfies the following additional properties, which include a min-max characterization as in Heinze \cite{heinze}, Heinze, Papanicolau and Stevens \cite{heinze-papanicolau-stevens}:
\begin{theorem}\label{THEOREM_W_speed}
Assume that \ref{asu_unbalancedW}, \ref{asu_projection_inverse}, \ref{asu_conv_mon} and \ref{asu_bc} hold. Let $(c^\star,\Fu)$ be a solution given by Theorem \ref{THEOREM_W_main}. Then, for any $\tilde{\Fu} \in \overline{\CS}$ such that $E_{c^\star}(\tilde{\Fu})=0$ we have that $(c^\star,\tilde{\Fu})$ solves \eqref{PROFILE_eq_Thm} and
\begin{equation}
c^\star= \frac{-\Fh}{\int_\R \lvert \tilde{\Fu}' \rvert^2}.
\end{equation}
In particular, the quantity $\int_\R \lvert \tilde{\Fu}' \rvert^2$ is well-defined and constant among the class of global minimizers of $E_{c^\star}$ in $\overline{\CS}$. Moreover, it holds
\begin{equation}
c^\star \leq \frac{\sqrt{-2\Fh}}{d} < \eta,
\end{equation}
with $d$ as in \eqref{d}, where we have also used \ref{asu_bc}. Finally, we have the variational characterization
\begin{equation}
c^\star= \sup \{c>0: \inf_{v \in \overline{\CS}}E_c(v)=-\infty \} = \inf\{c>0: \inf_{v \in \overline{\CS}}E_c(v)=0 \}.
\end{equation}
\end{theorem}
\begin{remark}
Combining the uniqueness part in Theorem \ref{THEOREM_W_main} and Theorem \ref{THEOREM_W_speed}, we obtain that for any $c>c^\star$, the infimum
\begin{equation}
\inf_{v \in \overline{\CS}}E_c(v)=0
\end{equation}
is never attained.
\end{remark}
\section{Extensions and open problems}
\subsection{Other classes of potentials}
Applying directly the abstract result from \cite{oliver-bonafoux-tw}, it is possible to replace \ref{asu_conv_mon} by an assumption which consists on an upper bound on $\Fh$. Following the idea in \cite{oliver-bonafoux-tw}, a particular family of potentials satisfying such an assumption can be obtained by suitably perturbing a given multi-well potential (or even a more general potential). More precisely, one can consider a non-negative smooth potential $V$ vanishing in two sets $\CA_1$ and $\CA_2$ in $\R^k$ and with local properties as in \ref{asu_unbalancedW} and \ref{asu_projection_inverse}, with the obvious modifications. A suitable perturbation around $\CA_2$ gives a family of unbalanced potentials $(W_\delta)_{\delta  \geq 0}$ such that $W_{0}=V$ and for $\delta \in (0,\delta_0]$ (for some $\delta_0>0$), the abstract framework from \cite{oliver-bonafoux-tw} applies to $W_\delta$ (up to an additive constant). In particular, one has results as the ones of this paper for $W_\delta$, which was a type of potential not allowed in \cite{alikakos-katzourakis}.
\subsection{The problem of traveling waves for 2D parabolic Allen-Cahn systems}
Our initial motivation comes from the study of the parabolic Allen-Cahn system
\begin{equation}\label{parabolic-allencahn_W}
\partial_t w - \Delta w = -\nabla_\bu V(w) \mbox{ in } [0,+\infty) \times \R^2,
\end{equation}
where $V$ is a standard, smooth, multi-well potential. In \cite{oliver-bonafoux-tw}, we show by an abstract approach that there exist traveling wave solutions $w:(t,x_1,x_2) \in [0,+\infty) \times \R^2 \to \FU(x_1-c^\star t,x_2) \in \R^k$ of \eqref{parabolic-allencahn_W} such that
\begin{equation}
\lim_{x_1 \to \pm +\infty}\lVert \FU(x_1,\cdot)-\Fq^\pm \rVert_{H^1(\R,\R^k)}=0
\end{equation}
where $\Fq^\pm$ solve
\begin{equation}
\left(\Fq^\pm\right)''=\nabla_\bu V(\Fq^\pm) \mbox{ in } \R
\end{equation}
and satisfy
\begin{equation}
\lim_{t \to \pm \infty} \Fq^\pm(t)=\sigma^\pm
\end{equation}
with $\sigma^\pm \in \Sigma:= \{ \bu \in \R^k: V(\bu)=0\}$. Moreover, $\Fq^-$ and $\Fq^+$ are (in some sense) \textit{local minimizers} of the energy
\begin{equation}
E: q \in H^1_\loc(\R,\R^k) \to E(q):= \int_\R \left[ \frac{\lvert q'(t) \rvert^2}{2}+V(q(t)) \right] dt
\end{equation}
with \textit{different energy levels}. That is
\begin{equation}
E(\Fq^\pm)=\Fm^\pm
\end{equation} 
with $\Fm^+> \Fm^-$ and $\Fm^-$ the global minima. The main point is that $\Fm^+>\Fm^+$, as the case $\Fm^+=\Fm^-$ had already been treated by several authors and it is by now well understood. See Fusco \cite{fusco}, Monteil and Santambrogio \cite{monteil-santambrogio}, Schatzman \cite{schatzman}, Smyrnelis \cite{smyrnelis}. As already explained, in \cite{oliver-bonafoux-tw}  we have to assume that $0 < \Fm^+-\Fm^-$ is bounded above by a given constant. As in here, the approach of the proof is to deduce the result from an abstract setting. In an attempt to extend our result to other situations, we were able to replace the perturbation assumption by another one in the abstract setting. The other assumption is the abstract version of \ref{asu_conv_mon} and the one we use here (see \ref{hyp_levelsets_bis}). However, the inconvenient is that such condition does not seem to apply when dealing with traveling waves of the system \eqref{parabolic-allencahn_W}. Even for simple examples of the potential $V$, we cannot prove that such condition is satisfied, which does not allow us to apply the abstract result. The reason is that properties such as convexity of level sets and uniform strict monotonicity along segments are harder to verify in infinite dimensions. However, if an example of potential $V$ such that \ref{hyp_levelsets_bis} is satisfied could be found, then we would have the existence of a pair $(c^\star,\FU)$ satisfying the same properties than the solution obtained in \cite{oliver-bonafoux-tw}.

\subsection{Traveling waves between homoclinics}
One considers again the Allen-Cahn system
\begin{equation}\label{parabolic-allencahn_W_bis}
\partial_t w - \Delta w = -\nabla_\bu V(w) \mbox{ in } [0,+\infty) \times \R^2,
\end{equation}
with $V$ as in the previous paragraph. We use the same notation as above. This time, we consider the conditions at infinity
\begin{equation}
\lim_{x_1 \to \pm +\infty}\lVert \FU(x_1,\cdot)-\Fu^\pm \rVert_{H^1(\R,\R^k)}=0
\end{equation}
for the associated profile $\FU$. Here $\Fu^-$ is the constant homoclinic solution equal to $\sigma \in \Sigma$ and $\Fu^+$ is such that $E(\Fu^+)>0$ and it solves
\begin{equation}
\left( \Fu^+ \right)''=\nabla_{\bu}V(\Fu^+) \mbox{ in } \R
\end{equation}
satisfying
\begin{equation}
\lim_{t \to \pm \infty} \Fu^+(t) = \sigma.
\end{equation}
That is, $\Fu^+$ is a non-constant homoclinic emanating from $\sigma$. If $\Fu^+$ is locally minimizing and, moreover, it satisfies the assumptions given by some of our abstract results\footnote{That is, either the perturbative one that we use in \cite{oliver-bonafoux-tw} or the one that we give here.}, then the existence of a solution $(c^\star,\FU)$ is guaranteed. The fact that $\Fu^+ \not = \Fu^+$ implies that $\FU$ is not constant, and since $\Fu^-$ and $\Fu^+$ are homoclinic, then $\FU$ can be thought as a heteroclinic between homoclinics. Moreover, notice that such type of solution cannot exist in the stationary case, as stationary waves join two different 1D solutions which are global minimizers, which is not possible in the homoclinic case\footnote{Clearly, the only globally minimizing homoclinic is the constant one.}. Even though for the moment we have no explicit example of potential $V$ such that the necessary assumptions are fulfilled, a previous result by the author \cite{oliver-bonafoux} suggest that homoclinic solutions might exist in some situations.

\section{The abstract setting}

As we advanced in the introduction, the main results of this paper are obtained from exploiting the abstract setting introduced in \cite{oliver-bonafoux-tw}. The current setting only differs to that in \cite{oliver-bonafoux-tw} by the modification of some assumptions, but the structure is analogous. As a consequence, we shall be more brief here in the exposition. We begin by recalling the main notations. Let $\scrL$ be a Hilbert space with inner product $\langle \cdot,\cdot \rangle_{\scrL}$ and induced norm $\lVert \cdot \rVert_{\scrL}$. Let $\scrH \subset \scrL$ a Hilbert space with inner product $\langle \cdot,\cdot \rangle_{\scrH}$. We will take $\CE: \scrL \to (-\infty,+\infty]$ an \textit{unbalanced} potential. We impose a set of general assumptions on $\CE$. Those assumptions are an adaptation of those in \cite{alikakos-katzourakis} for the infinite-dimensional setting with non-isolated minima. We will begin by fixing two sets $\scrF^-$ and $\scrF^+$ in $\scrL$.  For $r>0$, we define
\begin{equation}\label{Fr}
\scrF^\pm_r:= \left\{ v \in \scrL: \inf_{\bv \in \scrF^\pm}\lVert v-\bv \rVert_{\scrL}\leq r \right\},
\end{equation}
and
\begin{equation}\label{FrH}
\scrF^\pm_{\scrH,r}:= \left\{ v \in \scrH: \inf_{\bv \in \scrF^\pm}\lVert v-\bv \rVert_{\scrH}\leq r \right\},
\end{equation}
that is, the closed balls in $\scrL$ and $\scrH$ respectively, with radius $r>0$ and center $\scrF^\pm$. The main assumption reads as follows:
\begin{hyp}\label{hyp_unbalanced}
The potential $\CE$ is weakly lower semicontinuous in $\scrL$. The sets $\scrF^-$ and $\scrF^+$ are closed in $\scrL$. There exists a constant $a<0$ such that
\begin{equation}
\forall v \in \scrL, \forall \bv^- \in \scrF^-, \hspace{2mm} \CE(v)  \geq \CE(\bv^-) = a
\end{equation}
and each $\bv^+ \in \scrF^+$ is a local minimizer satisfying $\CE(\bv^+) =0$. Moreover, there exist two positive constants $r_0^-$, $r_0^+$ such that $\scrF^+_{r_0^-} \cap \scrF^-_{r_0^+} = \emptyset$ (see \eqref{Fr}). There also exist $C^\pm>1$ such that
\begin{equation}\label{coercivityL}
\forall v \in \scrF^\pm_{r_0^\pm}, \hspace{2mm} (C^\pm)^{-1}\dist_{\scrL}(v,\scrF^\pm)^2 \leq \CE(v)-\min\{ \pm(-a),0\}.
\end{equation}
Moreover, for any $v \in \scrF_{r_0^\pm}^\pm$, there exists a unique $\bv^\pm(v) \in \scrF^\pm$ such that
\begin{equation}
\lVert v-\bv^\pm(v) \rVert_{\scrL}=\inf_{\bv^\pm \in \scrF^\pm}\lVert v-\bv^\pm \rVert_{\scrL}.
\end{equation}
Moreover, the projection maps
\begin{equation}\label{projection}
P^\pm: v \in \scrF^\pm_{r_0^\pm} \to \bv^\pm(v) \in \scrF^\pm
\end{equation}
are $C^2$ with respect to the $\scrL$-norm.
\end{hyp}
The abstract result consists essentially on establishing existence of a pair $(c,\bU)$ in $(0,+\infty) \times Y$ ($Y$ is a function space to be defined later, see \eqref{Y}) which fulfills
\begin{equation}\label{abstract_equation}
\bU''-D_{\scrL}\CE(\bU)=-c\bU' \mbox{ in } \R
\end{equation}
and satisfies the conditions at infinity
\begin{equation}\label{abstract_bc_weak}
\exists T^- \leq 0: \forall t \leq T^-, \hspace{2mm} \bU(t) \in \scrF^-_{r_0^-/2} \hspace{1mm}\mbox{ and }\hspace{1mm} \exists \bv^+(\bU) \in \scrF^+: \lim_{t \to +\infty} \lVert \bU(t)-\bv^+(\bU) \rVert_{\scrH}=0.
\end{equation}
Hypothesis \ref{hyp_unbalanced} defines $\CE$ as an unbalanced double well potential with respect to $\scrF^-$ and $\scrF^+$ and gives local information of the minimizing sets. We have the following immediate consequence, which will be useful in the sequel:
\begin{lemma}\label{LEMMA_positivity}
Assume that \ref{hyp_unbalanced} holds. If we define for $r \in (0, r_0^\pm]$ we define
\begin{equation}\label{kappa_r}
\kappa^\pm_r:=\inf\{ \CE(v): \dist_\scrL(v,\scrF^\pm) \in [r,r_0^\pm] \}
\end{equation}
then we have $\kappa^\pm_r>\min\{ \pm(-a),0\}$. Moreover,
\begin{equation}\label{positivity}
\forall v \in \scrF^+_{r_0^+/2},\hspace{2mm} \CE(v) \geq 0.
\end{equation}
\end{lemma}
\begin{proof}
It follows directly from \eqref{coercivityL} in \ref{hyp_unbalanced}.
\end{proof}
 We now impose the following regarding the relationship between $\scrL$ and $\scrH$:
\begin{hyp}\label{hyp_smaller_space}
We have that $\scrH=\{ v \in \scrL: \CE(v)<+\infty\}$ and $\lVert \cdot \rVert_{\scrL} \leq \lVert \cdot \rVert_{\scrH}$. In particular, $\scrF^\pm \subset \scrH$. Furthermore $\CE$ is coercive in $\scrH$ in the following weak sense: there exists a constant $C_{\textit{c}}$ such that
\begin{equation}\label{coercivity}
\forall v \in \scrH, \hspace{2mm} \lVert v \rVert_{\scrH}^2-\lVert v \rVert_{\scrL}^2 \leq C_{\textit{c}}(\CE(v)-a)
\end{equation}
Moreover, $\CE$ is a $C^1$ functional on $(\scrH,\lVert \cdot \rVert_{\scrH})$ with differential $D\CE:v \in \scrH \to D\CE(v) \in \scrH$ and
\begin{equation}\label{CE_R_est}
\forall R>0, \exists C_{\CE}(R)>0, \forall (v,w) \in \scrH^2: \lVert v \rVert_{\scrH} \leq R, \hspace{2mm} \lvert D\CE(v)(w)  \rvert \leq C_{\CE}(R)\lVert w \rVert_{\scrH}.
\end{equation}
Finally, there exists an even smaller space $\scrH'$ with an inner product $\langle \cdot, \cdot \rangle_{\tilde{\scrH}}$ and associated norm $\lVert \cdot \rVert_{\tilde{\scrH}} \geq \lVert \cdot \rVert_{\scrH}$ such that we can find a continuous correspondence
\begin{equation}\label{DL}
D_{\scrL}\CE:v \in (\tilde{\scrH},\lVert \cdot \rVert_{\tilde{\scrH}}) \to D_{\scrL}\CE(v) \in (\scrL,\lVert \cdot \rVert_{\scrL})
\end{equation}
such that
\begin{equation}\label{abstract_ipp}
\forall v \in \tilde{\scrH}, \forall w \in \scrH,\hspace{2mm} D_{\scrL}\CE(v)(w)= D\CE(v)(w),
\end{equation}
where we have identified the derivative of $\CE$ with an element of $\scrH$ via Riesz's Theorem.
\end{hyp}
We now continue by imposing a compactness assumption on $\scrF^\pm$.
\begin{hyp}\label{hyp_compactness}
$\scrL$-bounded subsets of $\scrF^\pm$ are compact with respect to $\scrH$-convergence.\footnote{hence, they are in particular compact with respect to $\scrL$-convergence}
\end{hyp}
Assumption \ref{hyp_compactness} readily implies the following:
\begin{lemma}\label{LEMMA_compactness}
Assume that \ref{hyp_unbalanced} and \ref{hyp_compactness} hold. Then, the sets $\scrF^\pm_{r_0^\pm/2}$ defined in \eqref{Fr} are closed in $\scrL$.
\end{lemma}

Assumption \ref{hyp_compactness} is necessary in order to establish the conditions at infinity. Subsequently, we impose the following:
\begin{hyp}\label{hyp_projections_inverse}
Assume that \ref{hyp_unbalanced} holds. One of the two holds:
\begin{enumerate}
\item $\scrF^\pm$ is bounded.
\item Fix $(v,\bv^\pm) \in \scrF^\pm_{r_0^\pm} \times \scrF^\pm$. There exist an associated map $\hat{P}^\pm_{(v,\bv^\pm)}: \scrL \to \scrL$ such that
\begin{equation}\label{hatP}
P^\pm(\hat{P}^\pm_{(v,\bv^\pm)}(v))=\bv^\pm
\end{equation}
and
\begin{equation}\label{hatP_dist}
 \dist_{\scrL}(\hat{P}^\pm_{(v,\bv^\pm)}(v),\scrF^\pm)=\dist_{\scrL}(v,\scrF^\pm).
\end{equation}
Moreover, $\hat{P}^\pm_{(v,\bv^\pm)}: \scrL \to \scrL$ is differentiable and
\begin{equation}\label{D_inverse}
\forall (w_1,w_2) \in \scrL^2, \hspace{2mm} \lVert D(\hat{P}^\pm_{(v,\bv^\pm)})(w_1,w_2) \rVert_{\scrL}=\lVert w_2 \rVert_{\scrL}
\end{equation}
\begin{equation}\label{hatP_CE}
\CE(\hat{P}_{(v,\bv^\pm)}(v))=\CE(v).
\end{equation}
\end{enumerate}
\end{hyp}
Essentially, we impose that the projections $P^\pm$ from  \ref{hyp_unbalanced} are, in some sense, invertible. We now impose an assumption for the sets $\scrF^\pm_{\scrH, r_0}$:
\begin{hyp}\label{hyp_projections_H}
For any $v \in \scrF_{\scrH,r_0^\pm}^\pm$, as defined in \eqref{FrH}, there exists a unique $\bv_\scrH^\pm(v) \in \scrF^\pm$ such that
\begin{equation}
\lVert v-\bv_\scrH^\pm(v) \rVert_{\scrL}=\inf_{\bv^\pm \in \scrF^\pm}\lVert v-\bv^\pm \rVert_{\scrL}.
\end{equation}
Moreover, the projection maps
\begin{equation}
P^\pm_\scrH: v \in \scrF^\pm_{\scrH,r_0^\pm} \to \bv_\scrH^\pm(v) \in \scrF^\pm
\end{equation}
are $C^1$ with respect to the $\scrH$-norm. Moreover, if $C^\pm>1$ is the constant from \ref{hyp_unbalanced}, we have
\begin{equation}\label{H_difference_projections}
\forall v \in \scrF^\pm_{\scrH,r_0^\pm}, \hspace{2mm} \lVert P^\pm(v)-P^\pm_{\scrH}(v) \rVert_{\scrH} \leq C^\pm \lVert v-P^\pm_{\scrH}(v) \rVert_{\scrH}.
\end{equation}
Furthermore, for each $r^\pm \in (0,r_0^\pm]$ there exist constants $\beta^\pm(r^\pm)>0$ such that in case that $v \in \scrF^\pm_{r_0^\pm}$ satisfies
\begin{equation}\label{H_local_bound}
\CE(v) \leq \min\{ \pm(-a),0\}+\beta^\pm(r^\pm),
\end{equation}
then $v \in \scrF_{\scrH,r}^\pm$. Finally, we have the following
\begin{equation}\label{H_local_est}
\forall v \in \scrF_{\scrH,r_0^\pm}^\pm, \hspace{2mm} (C^\pm)^{-2}\lVert v-P^\pm_\scrH(v) \rVert_{\scrH}^2 \leq \CE(v)- \min\{ \pm(-a),0\} \leq (C^\pm)^2 \lVert v-P^\pm_\scrH(v) \rVert_{\scrH}^2.
\end{equation}
\end{hyp}


Assumption \ref{hyp_projections_H} is made in order to ensure the suitable local properties around $\scrF^\pm$ in $\scrH$. Before introducing the last assumptions, we need some additional notation. For $U \in H^1_{\loc}(\R,\scrL)$ and $c>0$, we (formally) define
\begin{equation}\label{Ec}
\bE_c(U):= \int_\R \be_c(U)(t) dt:= \int_{\R} \left[ \frac{\lVert U'(t) \rVert_{\scrL}^2}{2}+\CE(U(t)) \right]e^{ct} dt.
\end{equation}
More generally, for $I \subset \R$ a non-empty interval and $U \in H^1_{\loc}(I,\scrL)$, put
\begin{equation}\label{Ec_interval}
\bE_c(U;I) := \int_I \be_c(U)(t) dt.
\end{equation}
Notice that the integrals defined in \eqref{Ec} and \eqref{Ec_interval} might not even make sense in general due to the fact that $\CE$ has a sign. Nevertheless, we can define the notion of \textit{local minimizer} of $\bE_c(\cdot;I)$ as follows:
\begin{definition}\label{def_local_minimizer}
Assume that \ref{hyp_unbalanced} and \ref{hyp_smaller_space} hold. Let $I \subset \R$ be a bounded, non-empty interval. Assume that  $U \in H^1_{\loc}(I,\scrL)$ is such that $E_c(U;I)$ is well-defined and finite. Assume also that there exists $C>0$ such that for any $\phi \in \CC^1_c(\mathrm{int}(I),(\scrH,\lVert \cdot \rVert_{\scrH}))$ such that
\begin{equation}
\max_{t \in I} \lVert \phi(t) \rVert_{\scrH}<C,
\end{equation}
the quantity $\bE_c(U+\phi;I)$ is well-defined and larger than $\bE_c(U;I)$. Then, we say that $U$ is a \emph{local minimizer} of $\bE_c(\cdot;I)$.
\end{definition}
We assume the following property for local minimizers:
\begin{hyp}\label{hyp_regularity}
Assume that \ref{hyp_unbalanced} and \ref{hyp_smaller_space} hold. There exists a map $\FP:\scrL \to \scrL$ such that
\begin{equation}\label{FP_1}
\forall v \in \scrL, \hspace{2mm} \CE(\FP(v)) \leq \CE(v) \mbox{ and } \CE(\FP(v))=\CE(v) \Leftrightarrow \FP(v)=v,
\end{equation}
\begin{equation}\label{FP_2}
\forall (v_1,v_2) \in \scrL^2, \hspace{2mm} \lVert \FP(v_1)-\FP(v_2) \rVert_{\scrL} \leq \lVert v_1-v_2 \rVert_{\scrL},
\end{equation}
and
\begin{equation}\label{FP_3}
\FP|_{\scrF^\pm}=\mathrm{Id}|_{\scrF^\pm}.
\end{equation}

 Let $I\subset \R$, possibly unbounded and non-empty. Let $c>0$. If $\bW \in H^1_{\loc}(I,\scrL)$ is a local minimizer of $\bE_c(\cdot;I)$ in the sense of Definition \ref{def_local_minimizer}, which, additionally, is such that for all $t \in I$, $\bW(t)=\FP(\bW(t))$,  then $\bW \in \CA(I)$ where for any open set $O\subset \R$, $\CA(O)$ is defined as
\begin{equation}\label{CA_def}
\CA(O):= \CC^2_\loc(O,\scrL) \cap \CC^1_\loc(O,(\scrH,\lVert \cdot \rVert_{\scrH})) \cap \CC^0_\loc(O,(\tilde{\scrH},\lVert \cdot \rVert_{\tilde{\scrH}}))
\end{equation}
and $\bW$ solves
\begin{equation}
\bW''-D_\scrL\CE(\bW)=-c\bW' \mbox{ in } I,
\end{equation}
where $D_\scrL\CE$ was introduced in \eqref{DL}. 
\end{hyp}
Before stating the abstract result, we introduce the following constants:
\begin{equation}\label{intr_eta_0-}
\eta_0^-:= \min\left\{ \sqrt{e^{-1}\frac{r_0^-}{4}\sqrt{2(\kappa^-_{r_0^-/4}-a)}}, \frac{r_0^-}{4} \right\}>0,
\end{equation}
\begin{equation}\label{intr_r-}
\hat{r}^-:= \frac{r_0^-}{C^-+1}>0
\end{equation}
\begin{equation}\label{intr_eps_0-}
\eps_0^-:= \frac{1}{(C^-)^2(C^-+1)}\min\left\{ \frac{(\eta_0^-)^2}{4},\kappa_{\eta_0^-}^--a,\beta^-(\hat{r}^-),\beta^-(\eta_0^-)\right\}>0,
\end{equation}
\begin{equation}\label{sfC}
\sfC^\pm:= \frac{1}{2}(C^\pm)^2((C^\pm)^2+(C^\pm+1)^2)>0,
\end{equation}
\begin{equation}\label{gamma-}
\gamma^-:= \frac{1}{C^-+\sfC^-}>0
\end{equation}
and
\begin{equation}\label{d0}
d_0:= \dist_{\scrL}(\scrF^+_{r_0^+/2},\scrF^-_{r_0^-/2})>0,
\end{equation}
where the constants $C^-, \beta^-(\hat{r}^-), \beta^-(\eta_0^-)$ are those from \ref{hyp_projections_H} and $\kappa_r^\pm$ for $r>0$ are defined in \eqref{kappa_r}. The fact that $d_0>0$ follows from Lemma \ref{LEMMA_compactness} and \ref{hyp_unbalanced}. We now replace the perturbation assumption in \cite{oliver-bonafoux-tw} for a new one, the generalization of \ref{asu_conv_mon}. We define for $\alpha \in \R$ the level sets
\begin{equation}
\CE^\alpha:= \{ v \in \scrH: \CE(v) \leq \alpha\}
\end{equation}
and notice that \ref{hyp_unbalanced} and \ref{hyp_smaller_space} imply that there exists $\alpha_0 >0$ such that
\begin{equation}\label{alpha_0}
\CE^{\alpha_0}=\CE^{\alpha_0,-} \cup \CE^{\alpha_0,+}
\end{equation}
with $\CE^{\alpha_0,\pm}$ disjoint, $\scrL$-closed and such that
\begin{equation}\label{alpha_0+}
\CE^{\alpha_0,+} \subset \scrF^+_{r_0^+/2}.
\end{equation}
As a consequence, for each $\alpha \in [a,\alpha_0]$ we can write
\begin{equation}
\CE^{\alpha}=\CE^{\alpha,-} \cup \CE^{\alpha,+}
\end{equation}
with $\CE^{\alpha,\pm}$ disjoint, $\scrL$-closed and such that if $\alpha_1 \leq \alpha_2$ in $[a,\alpha_0]$, then $\CE^{\alpha_1,\pm} \subset \CE^{\alpha_2,\pm}$. In particular, we have
\begin{equation}\label{alpha-}
\forall \alpha \in [a,\alpha_0], \hspace{2mm} \CE^{\alpha,-} \cap \scrF^+_{r_0^+/2} = \emptyset.
\end{equation}
Next, notice that $\CE^{0,+}=\scrF^+$ and if $\alpha \in [a,0)$, then $\CE^{\alpha,+}=\emptyset$. Our next assumption writes then as follows:
\begin{hyp}\label{hyp_levelsets_bis}
Assume that \ref{hyp_unbalanced} and \ref{hyp_smaller_space} hold. We suppose that for all $\alpha \in [a,\alpha_0]$, we hav that $ \CE^{\alpha,-}$ is convex. For all $\bv^- \in \scrF^-$, $\alpha \in [a,\alpha_0]$ and $v \in \CE^{\alpha,-}$ define
\begin{equation}\label{A_alpha_set}
A(\alpha,\bv^-,v):=\{ \lambda \in \R:\CE( \bv^-+\lambda (v-\bv^-)) > \alpha \mbox{ and } \bv^-+\lambda (v-\bv^-) \in \CE^{\alpha_0,-}\}.
\end{equation}
Then, there exists $\alpha_{-} \in (a,0)$ such that
\begin{enumerate}
\item $\CE^{\alpha_{-}}\subset \scrF^-_{r_0^-/2}$.
\item There exists $\omega>0$ such that
\begin{equation}\label{omega_def}
\forall v \in \CE^{\alpha_0,-}, \forall \bv^- \in \scrF^-, \forall \theta \in A\left(\frac{\alpha_{-}+a}{2},\bv^-,v\right), \hspace{2mm} \frac{d}{d\lambda}(\CE(\bv^-+\lambda (v-\bv^-)))(\theta)\geq \omega.
\end{equation}
\item We have that for all $\alpha \in (a,\alpha_-]$, there exists $\omega(\alpha)>0$ such that, for all $v \in \CE^{\alpha_0,-}$ such that $\CE(v) \geq \alpha$, it holds
\begin{equation}\label{local_monotonicity}
 \exists \delta(v) > 0, \forall \theta \in (1-\delta(v),1+\delta(v)), \hspace{2mm}  \frac{d}{d\lambda}(\CE(P^-(v)+\lambda (v-P^-(v))))(\theta)\geq \omega(\alpha),
\end{equation}
where $P^-$ is the projection map defined in \ref{hyp_unbalanced}, which is well defined in $\CE^{\alpha_-}$ because $\CE^{\alpha_-} \subset \scrF_{r_0^-/2}^-$.
\end{enumerate}
Finally, we have the following global bound property: There exists $R>0$ such that for any $v \in \CE^{\alpha_0,-}$ there exists $\bv^-(v) \in \scrF^-$ such that $\lVert v-\bv^- (v)\rVert_{\scrH} \leq R$.
\end{hyp} 
Let us next define the space
\begin{align}\label{X}
X:= \left\{ U \in H^1_{\loc}(\R,\scrL): \exists T \geq 1, \right. & 
\forall t \geq T, \hspace{2mm} \dist_{\scrL}(U(t),\scrF^+) \leq \frac{r_0^+}{2}, \\ &\left.
\forall t \leq -T, \hspace{2mm} \dist_{\scrL}(U(t),\scrF^-) \leq \frac{r_0^-}{2} \right\},
\end{align}
already introduced in \cite{oliver-bonafoux-tw}. If \ref{hyp_levelsets_bis} holds, we define the class
\begin{align}\label{Y}
Y:= \left\{ U \in H^1_{\loc}(\R,\scrL): \exists T \geq 1, \right. & 
\forall t \geq T, \hspace{2mm} \dist_{\scrL}(U(t),\scrF^+) \leq \frac{r_0^+}{2}, \\ &\left.
\forall t \leq -T, \hspace{2mm} \CE(U(t)) \leq \alpha_- \right\}
\end{align}
so that $Y \subset X$ by \ref{hyp_levelsets_bis}. For working under assumption \ref{hyp_levelsets_bis}, the space $Y$ becomes more suitable in certain aspects. We can now finally state the abstract result
\begin{theorem}[Main abstract result]\label{THEOREM-ABSTRACT}
Assume that \ref{hyp_compactness}, \ref{hyp_projections_inverse}, \ref{hyp_projections_H}, \ref{hyp_regularity} and \ref{hyp_levelsets_bis} hold. Then, it holds
\begin{enumerate}
\item \textbf{Existence}. There exists $c^\star>0$ and $\bU \in \CA(\R) \cap Y$, $\CA(\R)$ as in \eqref{CA_def}, such that $(c^\star,\bU)$ solves \eqref{abstract_equation} with conditions at infinity \eqref{abstract_bc_weak} and $\bU$ is a global minimizer of $\bE_c$ in $Y$ (that is, $\bE_c(\bU)=0$). 
\item \textbf{Uniqueness of the speed}. The speed $c^\star$ is unique in the following sense: if $\overline{c^\star}>0$ is such that
\begin{equation}
\inf_{U \in Y}\bE_{\overline{c^\star}}(U) =0
\end{equation}
and there exists $\overline{\bU} \in \CA(\R) \cap X$ such that $(\overline{c^\star},\overline{\bU})$ solves \eqref{abstract_equation} and $\bE_{\overline{c^\star}}(\overline{\bU})<+\infty$, then $\overline{c^\star}=c^\star$. 
\item \textbf{Exponential convergence}. We have that for some $M^+>0$ it holds for all $t \in \R$
\begin{equation}\label{exponential_abstract}
\lVert \bU(t)-\bv^+(U) \rVert_{\scrL} \leq M^+e^{-c^\star t}
\end{equation}
where $\bv^+(U)$ is given by \eqref{abstract_bc_weak}.
\end{enumerate}
\end{theorem}
If $\scrF^-$ is reduced to a simple point and \ref{hyp_levelsets_bis} holds, we can shown convergence with respect to the $\scrL$-norm as $t \to -\infty$:
\begin{theorem}\label{THEOREM_abstract_isolated}
Assume that \ref{hyp_compactness}, \ref{hyp_projections_inverse}, \ref{hyp_projections_H}, \ref{hyp_regularity} and \ref{hyp_levelsets_bis} hold. Assume moreover that $\scrF^-=\{\bv^-_i\}$. Then, if $(c^\star,\bU)$ is the solution given by Theorem \ref{THEOREM-ABSTRACT}, $\bU$ satisfies in addition
\begin{equation}\label{abstract_bc_L}
\lim_{t \to -\infty}\lVert \bU(t)-\bv^-_i \rVert_{\scrL}=0.
\end{equation}
\end{theorem}
\begin{remark}
Notice that \eqref{abstract_bc_L} \textit{does not} imply that $\CE(\bU(t)) \to a$ as $t \to -\infty$, since $\CE$ is only supposed to be lower-semicontinuous with respect to $\scrL$-convergence.
\end{remark}
We can impose an additional assumption in order to obtain stronger conditions at $- \infty$ on the solution:
\begin{hyp}\label{hyp_convergence}
Hypothesis \ref{hyp_levelsets_bis} is fulfilled and, additionally:
\begin{equation}\label{a_convergence_bis}
-a < \frac{(d_{\alpha_0}\gamma^-)^2}{2}
\end{equation}
where $\gamma^-$ is as in \eqref{gamma-} and
\begin{equation}\label{d_alpha_0}
d_{\alpha_0}:= \dist_{\scrL}(\CE^{\alpha_0,-},\scrF^+_{r_0^+/2})
\end{equation}
which is positive by \ref{hyp_levelsets_bis} and Lemma \ref{LEMMA_compactness}.
\end{hyp}
If \ref{hyp_convergence} holds, we can show the following:
\begin{theorem}\label{THEOREM_abstract_bc}
Assume that  \ref{hyp_compactness}, \ref{hyp_projections_inverse}, \ref{hyp_projections_H}, \ref{hyp_regularity} and \ref{hyp_convergence} hold. Finally, assume that \ref{hyp_convergence} holds. Then, if $(c^\star,\bU)$ is the solution given by Theorem \ref{THEOREM-ABSTRACT}, it holds that $c^\star<\gamma^-$, $\gamma^-$ as in \eqref{gamma-} and there exists $M^->0$ such that for all $t \in \R$
\begin{equation}\label{abstract_bc_strong}
\lVert \bU(t)-\bv^-(\bU) \rVert_{\scrL}  \leq M^-e^{(\gamma^--c^\star)t}
\end{equation}
for some $\bv^-(U) \in \scrF^-$.
\end{theorem}
Finally, we have the characterization of the speed:
\begin{theorem}\label{THEOREM_abstract_speed}
Assume that  \ref{hyp_compactness}, \ref{hyp_projections_inverse}, \ref{hyp_projections_H}, \ref{hyp_regularity} and \ref{hyp_convergence} hold. Finally, assume that \ref{hyp_convergence} holds. Let $(c^\star,\bU)$ be the solution given by Theorem \ref{THEOREM-ABSTRACT}. Then, if $\tilde{\bU} \in \CA(\R) \cap Y$ is such that
\begin{equation}
\bE_{c^\star}(\tilde{\bU}) =0
\end{equation}
then we have that $(c^\star,\tilde{\bU})$ solves \eqref{abstract_equation} and
\begin{equation}\label{abstract_speed_formula}
c^\star=\frac{-a}{\int_\R \lVert \tilde{\bU}'(t) \rVert_{\scrL}^2dt}.
\end{equation}
In particular, the quantity $\int_\R \lVert \tilde{\bU}'(t) \rVert_{\scrL}^2dt$ is finite. Moreover, we have that
\begin{equation}
c^\star \leq \frac{\sqrt{-2a}}{d_{\alpha_0}} <\gamma^-,
\end{equation}
where we have used \ref{hyp_convergence} and $d_{\alpha_0}$, $\gamma^-$ are as in \eqref{d_alpha_0} and \eqref{gamma-} respectively. Finally, we have
\begin{equation}\label{abstract_speed_variational}
c^\star=\sup\{c>0: \inf_{U \in Y}\bE_c(U) = -\infty\}=\inf\{ c>0: \inf_{U \in Y}\bE_c(U)=0\}.
\end{equation}
\end{theorem}
\section{Proof of the abstract results}
The strategy is essentially the same that the one used in \cite{oliver-bonafoux-tw}, which is inspired by  \cite{alikakos-katzourakis}. See also the book by Alikakos, Fusco and Smyrnelis \cite{alikakos-fusco-smyrnelis}. We reproduce the main steps and provide the proof of the new results, which are essentially those in subsection \ref{subs_comparison}, for which we use \ref{hyp_levelsets_bis} instead of the bound assumption that we made in \cite{oliver-bonafoux-tw}. Therefore, each result here has its equivalent one in \cite{oliver-bonafoux-tw} and in some cases the arguments adapt in a straightforward way or the result can even be applied directly.
\subsection{Preliminaries}
Let $r_0^-$ and $r_0^+$ be the constants introduced before and $\scrF_{r_0^\pm/2}^\pm$ be the corresponding closed balls as in \eqref{Fr}. Assume that \ref{hyp_unbalanced} holds. For $T \geq 1$, we define the sets
\begin{equation}
X_{T}^-:= \left\{ U\in H^1_{\loc}(\R,\scrL): \forall t \leq -T, \hspace{1mm} U(t) \in \scrF^-_{r_0^-/2} \right\},
\end{equation}
\begin{equation}
X_{T}^+:= \left\{ U \in H^1_{\loc}(\R,\scrL): \forall t \geq T, \hspace{1mm}U(t) \in \scrF^+_{r_0^+/2} \right\},
\end{equation}
where the constant . Subsequently,  we set
\begin{equation}\label{XT}
X_{T}:=X_{T}^- \cap X_{T}^+.
\end{equation}
In case \ref{hyp_levelsets_bis} holds, we can define for $T \geq 1$
\begin{equation}
Y_T^-:= \left\{ U\in H^1_{\loc}(\R,\scrL): \forall t \leq -T, \hspace{1mm} \CE(U(t)) \leq \alpha_- \right\},
\end{equation}
where $\alpha_-$ is the constant introduced in assumption \ref{hyp_levelsets_bis}.  Subsequently, we set
\begin{equation}\label{YT}
Y_T:= Y_T^- \cap X_T^+.
\end{equation}
The spaces $Y_T$ will play here the role that the spaces $X_T$ played in \cite{oliver-bonafoux-tw}. We have that
\begin{equation}
X=\bigcup_{T \geq 1}X_T
\end{equation}
and
\begin{equation}
Y=\bigcup_{T \geq 1}Y_T.
\end{equation}

The following property is immediate:
\begin{lemma}\label{LEMMA_YT}
Assume that  \ref{hyp_levelsets_bis} holds. Then, we have that
\begin{equation}\label{YT_inclusion}
Y_T \subset X_T
\end{equation}
\end{lemma}
\begin{proof}
From \ref{hyp_levelsets_bis} it follows that whenever $v \in \scrH$ is such that $\CE(v) \leq \alpha_-$, it then holds that $v \in \scrF^-_{r_0^-/2}$. The inclusion \eqref{YT_inclusion} then follows by the definitions \eqref{XT} and \eqref{YT}.
\end{proof}
Lemma \ref{LEMMA_YT} shows that whenever \ref{hyp_levelsets_bis} holds,  $Y_T$ is a subspace of $X_T$. Therefore, properties that hold for $X_T$ (and that we proved in \cite{oliver-bonafoux-tw}) will apply to $Y_T$. Next, we recall the following properties:
\begin{lemma}[\cite{oliver-bonafoux-tw}]\label{LEMMA_well_defined}
Assume that \ref{hyp_unbalanced} holds.  Let $c>0$ and $T \geq 1$. For any $U \in X_T$, we have that
\begin{equation}\label{positivity_U}
\forall t \geq T, \hspace{2mm} \CE(U(t)) \geq 0.
\end{equation}
Moreover, the quantity $\bE_c(U)$ as introduced in \eqref{Ec} is well defined in $(-\infty,+\infty]$.
\end{lemma}
\begin{lemma}[\cite{oliver-bonafoux-tw}]\label{LEMMA_limit+}
Assume that \ref{hyp_unbalanced} and \ref{hyp_projections_H} hold. Let $c>0$ and $T \geq 1$. Take $U \in X_T$ such that $\bE_c(U) <+\infty$. Then, we have that there exists a subsequence $(t_n)_{n \in \N}$ in $\R$ such that $t_n \to +\infty$ as $n \to \infty$ and
\begin{equation}\label{limit+_energy0_weak}
\lim_{n \to \infty} \CE(U(t_n))e^{c t_n}=0.
\end{equation}
Moreover, there exists $\bv^+(U) \in \scrF^+$ such that for all $t \in \R$ it holds
\begin{equation}\label{limit+_function0_weak}
\lVert U(t) - \bv^+(U) \rVert_{\scrL}^2 \leq \left(\frac{\bE_c(U) -\frac{a}{c}e^{cT}}{c}\right)e^{-ct}.
\end{equation}
That is, $U$ tends to $\bv^+(U)$ at $+\infty$ with an exponential rate of convergence and with respect to the $\scrL$-norm.
\end{lemma}
Define now the infimum value
\begin{equation}\label{bcT}
\bb_{c,T}:= \inf_{U \in Y_T} \bE_c(U).
\end{equation}
It will be later shown that $\bb_{c,T}$ is always attained. We first have the following:
\begin{lemma}\label{LEMMA_PSI_YT}
Assume that \ref{hyp_unbalanced} and \ref{hyp_smaller_space} hold. Fix $\hat{\bv}^\pm \in \scrF^\pm$. Let $c>0$ and $T \geq 1$. For all $T \geq 1$ the function
\begin{equation}\label{Psi}
\Psi(t):= \begin{cases}
\hat{\bv}^- &\mbox{ if } t \leq -1,\\
\frac{1-t}{2}\hat{\bv}^-+\frac{t+1}{2}\hat{\bv}^+ &\mbox{ if } -1 \leq t \leq 1,\\
\hat{\bv}^+ &\mbox{ if } t \geq 1,
\end{cases}
\end{equation}
belongs to $X_T$. If \ref{hyp_levelsets_bis} holds, we also have $\Psi \in Y_T$. Moreover, for all $c>0$
\begin{equation}\label{Ec_PSI}
\bE_c(\Psi)<+\infty.
\end{equation}
Furthermore, we have
\begin{equation}\label{bcT_finite}
-\infty<\bb_{c,T}<+\infty.
\end{equation}
\end{lemma}
The proof of Lemma \ref{LEMMA_PSI_YT} follows the same lines that the equivalent result in \cite{oliver-bonafoux-tw}, so we skip it.
\subsection{Two (semi-)continuity results}
We now recall and extend some results regarding continuity and semi-continuity properties of $\bE_c$. We first have:
\begin{lemma}[\cite{oliver-bonafoux-tw}]\label{LEMMA_ATU}
Assume that \ref{hyp_unbalanced} holds. Fix $T \geq 1$ and $U \in X_T$. Consider the set
\begin{equation}\label{ATU}
A_{T,U}:=\{ c>0: \bE_c(U) <+\infty\}.
\end{equation}
Then,  if $c \in A_{T,U}$, then $(0,c] \subset A_{T,U}$. Moreover, the correspondence
\begin{equation}
c \in A_{T,U} \to \bE_c(U) \in \R
\end{equation}
 is continuous.
\end{lemma}
Next, we extend a result from \cite{oliver-bonafoux-tw}:
\begin{lemma}\label{LEMMA_LSC_YT}
Assume that \ref{hyp_unbalanced}, \ref{hyp_compactness} and \ref{hyp_projections_inverse} hold. Let $T \geq 1$ be fixed. Let $(U_n^i)_{n \in \N}$ be a sequence in $X_T$ and $(c_n)_{n \in \N}$ a convergent sequence of positive real numbers such that
\begin{equation}\label{LSC_hyp}
\sup_{n \in \N}\bE_{c_n}(U_n^i) <+\infty.
\end{equation}
Then, there exists a sequence $(U_n)$ in $X_T$ and $U_\infty \in X_T$ such that up to extracting a subsequence in $(U_n,c_n)_{n \in \N}$ it holds
\begin{equation}\label{LSC_U_n}
\forall n \in \N, \hspace{2mm} \bE_c(U_n)=\bE_c(U_n^i),
\end{equation}
\begin{equation}\label{U_infty_conv1}
\forall t \in \R, \hspace{2mm} U_n(t) \rightharpoonup U_\infty(t) \mbox{ weakly in } \scrL
\end{equation}
\begin{equation}\label{U_infty_conv2}
U_n'e^{c_n \mathrm{Id}/2} \rightharpoonup U_\infty' e^{c_\infty \mathrm{Id}/2} \mbox{ weakly in } L^2(\R,\scrL)
\end{equation}
and
\begin{equation}\label{U_infty_LSC}
\bE_{c_\infty}(U_\infty) \leq \liminf_{n \to \infty}\bE_{c_n}(U_n),
\end{equation}
where $c_\infty:=\lim_{n \to \infty}c_n$. Moreover, if \ref{hyp_levelsets_bis} holds and the sequence $(U_n^i)_{n \in \N}$ is contained in $Y_T$, then $U_\infty \in Y_T$.
\end{lemma}
\begin{proof}
Except for the last part, the result is exactly the same as that in \cite{oliver-bonafoux-tw}. Notice that if $(U_n^i)_{n \in \N}$ is contained in $Y_T$, then by \eqref{U_infty_conv1} and the lower-semicontinuity of $\CE$ we have that for all $t \leq -T$ it holds
\begin{equation}
\CE(U_\infty(t)) \leq \liminf_{n \to \infty} \CE(U_n(t)) \leq \alpha_-
\end{equation}
which shows that $U_\infty \in Y_T$ and concludes the proof.
\end{proof}
\subsection{Existence of an infimum for $\bE_c$ in $Y_T$}
We now show that for any $c>0$ and $T \geq 1$, the infimum value $\bb_{c,T}$ defined in \eqref{bcT} is attained. The result is as follows:
\begin{lemma}\label{LEMMA_bcT}
Assume that \ref{hyp_compactness}, \ref{hyp_projections_inverse} and \ref{hyp_levelsets_bis} hold. Let $c>0$, $T \geq 1$ and $\bb_{c,T}$ be as in \eqref{bcT}. Then, $\bb_{c,T}$ is attained for some $\bV_{c,T} \in Y_T$. 
\end{lemma}
\begin{proof}
By \eqref{bcT_finite} in Lemma \ref{LEMMA_PSI_YT}, we have that there exists a minimizing sequence $(U_n)_{n \in \N}$ in $Y_T$. Therefore, Lemma \ref{LEMMA_LSC_YT} implies the existence of $\bV_{c,T} \in Y_T$. 
\end{proof}
We next show that constrained minimizers are solutions of the equation \eqref{abstract_equation} on a (possibly proper) subset of $\R$.
\begin{lemma}\label{LEMMA_bcT_solution}
Assume that \ref{hyp_regularity} and \ref{hyp_levelsets_bis} hold. Let $c>0$, $T \geq 1$ and $\bb_{c,T}$ be as in \eqref{bcT}. Let $\bV_{c,T} \in Y_T$ be such that $\bE_c(\bV_{c,T})=\bb_{c,T}$. Then, $\bV_{c,T} \in \CA((-T,T))$ and
\begin{equation}\label{bcT_equation}
\bV_{c,T}''-D_{\scrL}\CE(\bV_{c,T})=-c\bV_{c,T}' \mbox{ in } (-T,T).
\end{equation}
Moreover, if $t \geq T$ is such that
\begin{equation}\label{bcT_no_constrain_+}
\dist_{\scrL}(\bV_{c,T}(t),\scrF^+)< \frac{r_0^+}{2},
\end{equation}
then, there exists $\delta^+_Y(t)>0$ such that $\bV_{c,T} \in \CA((t-\delta^+_Y(t),t+\delta^+_Y(t)))$ and
\begin{equation}\label{bcT_equation_delta+}
\bV_{c,T}''-D_{\scrL}\CE(\bV_{c,T})=-c\bV_{c,T}' \mbox{ in } (t-\delta^+_Y(t),t+\delta^+_Y(t)).
\end{equation}
Similarly, if $t \leq -T$ is such that there exist $\eps^-_Y(t)>0$ and $\delta^-_Y(t)>0$ such that
\begin{equation}\label{bcT_no_constrain_-}
\forall s \in [t-\delta^-_Y(t),t+\delta^-_Y(t)], \hspace{2mm} \CE(\bV_{c,T}(s)) \leq \alpha_--\eps^-_Y(t)
\end{equation}
then, $\bV_{c,T} \in \CA((t-\delta^-_Y(t),t+\delta^-_Y(t)))$ and
\begin{equation}\label{bcT_equation_delta-}
\bV_{c,T}''-D_{\scrL}\CE(\bV_{c,T})=-c\bV_{c,T}' \mbox{ in } (t-\delta^-_Y(t),t+\delta^-_Y(t)).
\end{equation}
\end{lemma}
\begin{proof}
We begin by showing that
\begin{equation}\label{bounded_bcT}
\forall t \in \R, \hspace{2mm} \bV_{c,T}(t)=\FP(\bV_{c,T}(t)).
\end{equation}
Arguing as in the proof of the equivalent result in \cite{oliver-bonafoux-tw}, to prove \eqref{bounded_bcT} it suffices to show that
\begin{equation}
\bV_{c,T}^\FP: t \in \R \to \bV_{c,T}^\FP(t)
\end{equation}
belongs to $Y_T$. Indeed, the proof given in \cite{oliver-bonafoux-tw} implies that $\bV_{c,T}^\FP \in X_T$. Thus, the fact that $\bV_{c,T} \in \CA((-T,T))$ and $\bV_{c,T}$ solves \eqref{bcT_equation} is proven as in \cite{oliver-bonafoux-tw}. Next, notice that \eqref{FP_1} implies that for all $t \leq -T$, we have $\CE(\bV_{c,T}^\FP(t))\leq \CE(\bV_{c,T}(t))\leq \alpha_-$. Therefore, $\bV_{c,T}^\FP \in Y_T$, which shows \eqref{bounded_bcT}. Let now $t \geq T$ such that \eqref{bcT_no_constrain_+} holds. Again, arguing as \cite{oliver-bonafoux-tw}, we obtain that there exists $\delta^+_Y(t)>0$ such that $\bV_{c,T} \in \CA((t-\delta^+_Y(t),t+\delta^+_Y(t)))$ and \eqref{bcT_equation_delta+} holds. To conclude the proof, let $t \leq -T$ be such that \eqref{bcT_no_constrain_-} holds. Notice that $\CE(\bV_{c,T})$ is finite on $(-\infty,T]$, which by assumption \ref{hyp_smaller_space} implies that $\bV_{c,T}$ takes values in $\scrH$ on $(-\infty,T]$. Moreover $\bV_{c,T}$ is $\scrL$-continuous. Therefore, using \eqref{coercivity} we get
\begin{equation}\label{R1_ineq1}
\forall s \in [t-\delta^-_Y(t),t+\delta^-_Y(t)], \hspace{2mm} \lVert \bV_{c,T}(s) \rVert_{\scrH} \leq R_1,
\end{equation}
with
\begin{equation}
0<R_1:= \left( C_c(\alpha_--a) + \max_{s \in [t-\delta^-_Y(t),t+\delta^-_Y(t)]}\lVert \bV_{c,T}(s) \rVert_{\scrL} \right)^{\frac{1}{2}}<+\infty.
\end{equation}
Let $\phi \in\CC_c^1((t-\delta^-_Y(t),t+\delta^-_Y(t)),(\scrH,\lVert \cdot \rVert_{\scrH}))$ be such that
\begin{equation}\label{R1_ineq2}
\max_{s \in [t-\delta^-_Y(t),t+\delta^-_Y(t)]} \lVert \phi(s) \rVert_{\scrH}\leq R_1,
\end{equation}
by \eqref{CE_R_est}, we have that if $v \in \scrH$ is such that $\lVert v \rVert_{\scrH} \leq 3R_1$, it holds
\begin{equation}\label{R1_ineq3}
\forall s \in [t-\delta^-_Y(t),t+\delta^-_Y(t)], \hspace{2mm} \lvert D\CE(v)(\phi(s)) \rvert \leq C_\CE(3R_1) \lVert \phi(s) \rVert_{\scrH}
\end{equation}
with $C_\CE(3R_1)$ independent on $v$. We have by \eqref{bcT_no_constrain_-} that
\begin{equation}\label{R1_expansion}
\forall s \in [t-\delta^-_Y(t),t+\delta^-_Y(t)], \hspace{2mm} \CE(\bV_{c,T}(s)+\phi(s)) \leq  \alpha_--\eps^-_Y(t)+D\CE(h(s))(\phi(s)),
\end{equation}
where, for all $s \in [t-\delta^-_Y(t),t+\delta^-_Y(t)]$, $h(s) \in \scrH$ lies on the segment joining $\bV_{c,T}(s)+\phi(s)$ and $\bV_{c,T}(s)$, so that $\lVert h(s) \rVert_{\scrH} \leq 3R_1$ by \eqref{R1_ineq1} and \eqref{R1_ineq2}. Therefore, we can plug \eqref{R1_ineq3} into \eqref{R1_expansion} to obtain
\begin{equation}\label{R1_ineq4}
\forall s \in [t-\delta^-_Y(t),t+\delta^-_Y(t)], \hspace{2mm} \CE(\bV_{c,T}(s)+\phi(s)) \leq \alpha_--\eps^-_Y(t)+C_\CE(3R_1) \lVert \phi(s) \rVert_{\scrH},
\end{equation}
meaning that if we choose any $\phi \in\CC_c^1((t-\delta^-_Y(t),t+\delta^-_Y(t)),(\scrH,\lVert \cdot \rVert_{\scrH}))$ such that
\begin{equation}
\max_{s \in [t-\delta^-_Y(t),t+\delta^-_Y(t)]} \lVert \phi(s) \rVert_{\scrH}\leq \min\left\{R_1, \frac{\eps_Y^-(t)}{2C_\CE(3R_1)} \right\}
\end{equation}
we get by \eqref{R1_ineq4} that
\begin{equation}
\forall s \in [t-\delta^-_Y(t),t+\delta^-_Y(t)], \hspace{2mm} \CE(\bV_{c,T}(s)+\phi(s)) \leq \alpha_--\frac{\eps^-_Y(t)}{2}<\alpha_-
\end{equation}
so that $\bV_{c,T}+\phi \in Y_T$. Therefore, we have that $\bV_{c,T}$ is a local minimizer of $\bE_c(\cdot;[t-\delta^-_Y(t),t+\delta^-_Y(t)])$ in the sense of Definition \ref{def_local_minimizer}. Since, in addition, $\bV_{c,T}$ satisfies \eqref{bounded_bcT}, we can apply \ref{hyp_regularity} and we obtain that $\bV_{c,T} \in \CA((t-\delta^-_Y(t),t+\delta^-_Y(t)))$ and \eqref{bcT_equation_delta-} holds. As a consequence, the proof is finished.
\end{proof}
\subsection{The comparison result}\label{subs_comparison}
As in \cite{oliver-bonafoux-tw}, we establish the necessary properties regarding the behavior of the constrained minimizers. This is done by using the fact that the constrained minimizers solve the ODE system along with comparison arguments. Essentially, we adapt the proof of the analogous result from \cite{alikakos-katzourakis,alikakos-fusco-smyrnelis} in the framework of \ref{asu_conv_mon}, which makes it more involved but uses the same ideas. We begin by recalling the definition of some useful constants. For $0<r\leq r_0^\pm$, recall the definition of $\kappa_r^\pm$ introduced in \eqref{kappa_r}, Lemma \ref{LEMMA_positivity}. We define
\begin{equation}\label{eta_0+}
\eta^+_0:= \min\left\{\sqrt{e^{-1}\frac{r_0^+}{4}\sqrt{2\kappa^+_{r_0^+/4}}}, \frac{r_0^+}{4} \right\}>0,
\end{equation}
\begin{equation}\label{r+}
\hat{r}^+:= \frac{r_0^+}{C^++1}>0,
\end{equation}
\begin{equation}\label{eps_0+}
\eps^+_0:= \frac{1}{\left(C^+\right)^2(C^++1)}\min\left\{\frac{(\eta_0^+)^2}{4}, \kappa^+_{\eta_0^+},\beta^+(\hat{r}^+),\beta^+(\eta_0^+)\right\}>0,
\end{equation}
where the constants $C^\pm$, $\beta^\pm(\hat{r}^\pm), \beta(\eta_0^\pm)$ were introduced in \ref{hyp_projections_H}. Recall that in \eqref{intr_eta_0-}, \eqref{intr_r-}, \eqref{intr_eps_0-} we introduced the analogous constants
\begin{equation}\label{eta_0-}
\eta_0^-:= \min\left\{\sqrt{e^{-1}\frac{r_0^-}{4}\sqrt{2(\kappa^-_{r_0^-/4}-a)}}, \frac{r_0^-}{4} \right\}>0,
\end{equation}
\begin{equation}\label{r-}
\hat{r}^-:= \frac{r_0^-}{C^-+1}>0
\end{equation}
and
\begin{equation}\label{eps_0-}
\eps_0^-:= \frac{1}{(C^-)^2(C^-+1)}\min\left\{ \frac{(\eta_0^-)^2}{4},\kappa_{\eta_0^-}^--a,\beta^-(\hat{r}^-),\beta^-(\eta_0^-)\right\}>0.
\end{equation}
For any $U \in X_T$, define
\begin{equation}\label{comparison_t-}
t^-(U,\eps^-_0):= \sup\left\{ t \in \R: \CE(U(t))\leq a+\eps^-_0 \mbox{ and } \dist_{\scrL}(U(t),\scrF^-) \leq \frac{r_0^-}{2} 	\right\}
\end{equation}
and
\begin{equation}\label{comparison_t+}
t^+(U,\eps^+_0):= \inf\left\{ t \in \R: \CE(U(t))\leq \eps^+_0 \mbox{ and } \dist_{\scrL}(U(t),\scrF^+) \leq \frac{r_0^+}{2} \right\}.
\end{equation}

We begin by proving two preliminary comparison results:
\begin{lemma}\label{LEMMA_comparison_bis_1}
Assume that \ref{hyp_levelsets_bis} holds. Let  $t_1 < t_2$ and $U \in H^1([t_1,t_2],\scrL)$. Assume that
\begin{equation}\label{comparison_bis_asu_1}
U(t_1), U(t_2) \in \CE^{\alpha,-} \mbox{ with } \alpha \leq 0
\end{equation}
and
\begin{equation}\label{comparison_bis_asu_2}
\exists t_3 \in (t_1,t_2): \CE(U(t_3)) > \alpha.
\end{equation}
Then, there exists $\tilde{U} \in H^1([t_1,t_2],\scrL)$ such that
\begin{equation}\label{comparison_bis_conclusion_1}
\forall t \in (t_1,t_2), \forall i \in \{1,2\}, \hspace{2mm} U(t) \in \CE^{\alpha,-}, \hspace{1mm} \tilde{U}(t_i)=U(t_i)
\end{equation}
and
\begin{equation}\label{comparison_bis_conclusion_2}
\forall c > 0, \hspace{2mm} \bE_c(\tilde{U}) < \bE_c(U).
\end{equation}
\end{lemma}
The following is inspired by similar results from \cite{alikakos-katzourakis}, although our version is weaker:
\begin{lemma}\label{LEMMA_comparison_bis_2}
Assume that \ref{hyp_levelsets_bis} holds. Let $t_1<t_2$, $\alpha \in \left(\frac{\alpha_-+a}{2},\alpha_0\right)$  and $U_\alpha$ be the constant function
\begin{equation}\label{comparison_bis_2_asu}
U_\alpha: t\in [t_1,t_2] \to v_\alpha \in \CE^{\alpha,-}, \hspace{2mm} \CE(v_\alpha)=\alpha.
\end{equation}
Then, there exists $\tilde{U} \in H^1_{\loc}([t_1,t_2],\scrL)$ such that
\begin{equation}\label{comparison_bis_2_conclusion_1}
\forall t \in (t_1,t_2), \forall i \in \{1,2\}, \hspace{2mm} \CE(\tilde{U}(t))<\alpha, \hspace{1mm} \tilde{U}(t_i)=v_\alpha
\end{equation}
and
\begin{equation}\label{comparison_bis_2_conclusion_2}
\forall c>0, \hspace{2mm} \bE_c(\tilde{U})<\bE_c(U_\alpha).
\end{equation}
\end{lemma}
Before proving Lemmas \ref{LEMMA_comparison_bis_1} and \ref{LEMMA_comparison_bis_2}, we recall the following property:
\begin{lemma}\label{LEMMA_convex}
Let $C \subset \scrL$ be a closed convex set of $\scrL$ and $P_C$ the corresponding orthogonal projection. If $I \subset \R$ is a non-empty interval and $U \in H^1_{\loc}(I,\scrL)$, it holds that $P_C(U(\cdot)) \in H^1_{\loc}(I,\scrL))$ and
\begin{equation}\label{convex_ineq}
\mbox{ for a. e. }  t \in I, \hspace{2mm}\lVert P_C(U(t))' \rVert_{\scrL} \leq \lVert U(t) \rVert_{\scrL}.
\end{equation}
\end{lemma}
\begin{proof}
The Hilbert projection  Theorem readily implies that $P_C$ is 1-Lipschitz. As a consequence, we have that $P_C(U(\cdot)) \in H^1_{\loc}(I,\scrL))$. Fix now $t \in I$, for any $s \in I \setminus \{ t\}$ we have
\begin{equation}
\left\lVert \frac{P_C(U(s))-P_C(U(t))}{s-t} \right\rVert_{\scrL} \leq \left\lVert \frac{U(s)-U(t)}{s-t} \right\rVert_{\scrL},
\end{equation}
which gives \eqref{convex_ineq} by applying Lebesgue's differentiation Theorem.
\end{proof}
\begin{proof}[Proof of Lemma \ref{LEMMA_comparison_bis_1}]
From \ref{hyp_levelsets_bis}, we have that $\CE^{\alpha,-}$ is closed and convex in $\scrL$. As a consequence, Lemma \ref{LEMMA_convex} applies to $\CE^{\alpha,-}$. Therefore, we can consider the corresponding orthogonal projection $P^{\alpha,-} : \scrL \to \CE^{\alpha,-}$. Define now
\begin{equation}
\forall t \in [t_1,t_2],\hspace{2mm} \tilde{U}(t):=P^{\alpha,-}(U(t)) 
\end{equation}
which by Lemma \ref{LEMMA_convex} belongs to $H^1([t_1,t_2],\scrL)$. Moreover, applying \eqref{comparison_bis_asu_1} we see that $\tilde{U}$ verifies \eqref{comparison_bis_conclusion_1}. Now, notice that by \ref{hyp_unbalanced} implies that $t \in [t_1,t_2] \to \CE(U(t))$ is lower-semicontinuous. Therefore, $\{t \in [t_1,t_2]: \CE(U(t)) > \alpha\}$ is open, and by \eqref{comparison_bis_asu_2} we have that it is also non-empty. As a consequence, there exists $I \subset [t_1,t_2]$ non-empty an open such that
\begin{equation}
\forall t \in I, \hspace{2mm} \CE(U(t)) > \alpha
\end{equation}
which implies that
\begin{equation}
\int_I \CE(\tilde{U}(t))e^{ct}dt < \int_I \CE(U(t))e^{ct}dt
\end{equation}
and, since for all $v \in \scrH$ we have $\CE(P^{\alpha,-}(v)) \leq \CE(v)$ by definition of the orthogonal projection and because $\alpha \leq 0$, we conclude
\begin{equation}
\int_{t_1}^{t_2} \CE(\tilde{U}(t))e^{ct}dt < \int_{t_1}^{t_2} \CE(U(t))e^{ct}dt,
\end{equation}
which in combination with \eqref{convex_ineq} in Lemma \ref{LEMMA_convex} gives \eqref{comparison_bis_conclusion_2} and concludes the proof.
\end{proof}
\begin{proof}[Proof of Lemma \ref{LEMMA_comparison_bis_2}]
Let $\eps>0$, $\bv^- \in \scrF^-$ and $\phi \in \CC^1_c([t_1,t_2])$ such that $\phi>0$ in $(t_1,t_2)$. Define the function
\begin{equation}
\forall t \in [t_1,t_2], \hspace{2mm} U_\eps(t):= \bv^-+(1-\eps\phi(t))(v_\alpha-\bv^-)
\end{equation}
so that $U_0=U_\alpha$. It is clear that for all $\eps>0$ we have $U_\eps \in H^1([t_1,t_2],\scrL)$. Moreover, we also have that for $i \in \{1,2\}$, it holds $U_\eps(t_i)=v_\alpha$. For all $t \in [t_1,t_2]$, consider $A\left(\frac{\alpha_-+a}{2},\bv^-,U(t)\right)$ as in \eqref{A_alpha_set}. Since $\alpha \in \left( \frac{\alpha_-+a}{2},\alpha_0^-\right)$, $v_\alpha \in \CE^{\alpha,-}$ and $\lVert \phi \rVert_{C^1([t_1,t_2])}<+\infty$, there exists $\eps_1>0$ independent on $t$ such that
\begin{equation}
(1-\eps_1 \phi(t),1+\eps_1 \phi(t)) \subset A\left(\frac{\alpha_-+a}{2},\bv^-,U(t)\right).
\end{equation}
Therefore, using \eqref{omega_def} we have that
\begin{align}
\forall \eps \in (0,\eps_1), \forall t \in [t_1,t_2], \hspace{2mm} \CE(U_\eps(t))-\CE(U_0(t)) &=- \int_{1-\eps \phi(t)}^1 \frac{d}{d\lambda}(\CE(\bv^-+\lambda(v_\alpha-\bv^-)) )(\theta)d\theta\\
& \leq -\eps \phi(t)\omega,\label{eps_1}
\end{align}
where $\omega>0$. In particular, we have that for all $\eps \in (0,\eps_1)$, $U_\eps$ fulfills \eqref{comparison_bis_2_conclusion_1}. Now, notice that
\begin{equation}
\forall \eps>0, \mbox{ for a. e. } t \in [t_1,t_2], \hspace{2mm} U_\eps'(t)=-\eps \phi'(t) (v_\alpha-\bv^-)
\end{equation}
which means that, using also \eqref{eps_1}, we obtain
\begin{align}
\forall c >0, &\forall \eps \in (0,\eps_1),\\ &\bE_c(U_\eps)-\bE_c(U_0) \leq (-\eps \lVert \phi \rVert_{L^\infty([t_1,t_2])}\omega+\frac{1}{2}\eps^2 \lVert \phi' \rVert_{L^\infty([t_1,t_2])}^2\lVert v_\alpha-\bv^- \rVert_{\scrL}^2)\frac{e^{ct_2}-e^{ct_1}}{c}
\end{align}
so that if $\eps \in (0,\eps_2)$, where
\begin{equation}
\eps_2:= \frac{ 2\lVert \phi \rVert_{L^\infty([t_1,t_2])}\omega}{\lVert \phi' \rVert_{L^\infty([t_1,t_2])}^2\lVert v_\alpha-\bv^- \rVert_{\scrL}^2}>0
\end{equation}
then \eqref{comparison_bis_2_conclusion_2} holds for $U_\eps$. Hence, we have shown that if $\eps \in (0,\eps_3)$ with $\eps_3:= \min \{\eps_1,\eps_2\}>0$, then $U_\eps$ fulfills \eqref{comparison_bis_2_conclusion_1} and \eqref{comparison_bis_2_conclusion_2}, which completes the proof.
\end{proof}
We will now define for $U \in Y_T$
\begin{equation}\label{comparison_t2}
t_2^-(U):= \inf\{ t \in \R: 	\CE(U(t)) = \alpha_0,  U(t) \in \CE^{\alpha_0,-}\}
\end{equation}
and
\begin{equation}\label{comparison_t1}
t_1^-(U):= \sup\{ t \leq t_2^-(U): \CE(U(t)) \leq \alpha_-\}
\end{equation}
where the constants $a < \alpha_- < 0 < \alpha_0$ are those from \ref{hyp_levelsets_bis}. The following preliminary property is straightforward:
\begin{lemma}\label{LEMMA_comparison_bis_preliminary}
Let $U \in Y_T$ and $t_2^-(U)$, $t_1^-(U)$ be as in \eqref{comparison_t2} and \eqref{comparison_t1} respectively. Then, it holds
\begin{equation}\label{comparison_bis_preliminary}
-T \leq t_1^-(U) \leq  t_2^-(U)
\end{equation}
\end{lemma}
\begin{proof}
We have that $\CE(U(-T)) \leq \alpha_-$, which implies that $-T \leq t_1^-(U)$. The inequality $t_1^-(U) \leq t_2^-(U)$ follows from the definition.
\end{proof}

Lemmas \ref{LEMMA_comparison_bis_1} and \ref{LEMMA_comparison_bis_2} apply mainly to obtain the following information on constrained minimizers in $Y_T$:
\begin{proposition}\label{PROPOSITION_comparison_bis_1}
Assume that \ref{hyp_compactness}, \ref{hyp_projections_inverse}, \ref{hyp_regularity} and \ref{hyp_levelsets_bis} hold. Let $c>0$ and $T \geq 1$. Let $\bV_{c,T}$ be a constrained minimizer of $\bE_c$ in $Y_T$ given by Lemma \ref{LEMMA_bcT}. Let $t_1^-:= t_1^-(\bV_{c,T})$ and $t_2^-:=t_2^-(\bV_{c,T})$, where $ t_1^-(\bV_{c,T})$ and $ t_2^-(\bV_{c,T})$ are as in \eqref{comparison_t1}, \eqref{comparison_t2} respectively. Then, it holds that $\CE(\bV_{c,T}(t_2^-))=\alpha_0$ and
\begin{equation}\label{sfT1}
t_2^--t_1^- \leq \sfT_1(c),
\end{equation}
where
\begin{equation}\label{sfT1_definition}
\sfT_1(c):= \frac{2Rc + 2\sqrt{R^2c^2+2R\omega}}{\omega}.
\end{equation}
Moreover, we have that
\begin{equation}\label{comparison_bis_1_t1}
\forall t \leq t_1^-, \hspace{2mm} \CE(\bV_{c,T}(t)) \leq \alpha_-
\end{equation}
and there exist $\delta_0^->0$ and $t_0^- \leq t_1^-$ such that
\begin{equation}\label{comparison_bis_1_t0}
\begin{cases}
\CE(\bV_{c,T}(\cdot)) \mbox{ is monotone in } (-\infty,t_0^-+\delta_0^-] \mbox{ and }\\
\forall t \leq t_0^-+\delta_0^-, \hspace{2mm} \CE(\bV_{c,T}(t)) < \alpha_-.
\end{cases}
\end{equation}
\end{proposition}
\begin{proof}
Notice that due to the definition of $t_2^-$ in \eqref{comparison_t2} for all $t \leq t_2^-$ we have $\bV_{c,T}(t) \in \CE^{\alpha_0,-}$. Therefore, $t_2^- < T$.
Notice now that if $t_1^-=t_2^-$, then \eqref{sfT1_definition} holds trivially.  Assume then that $t_1^- < t_2^-$. By Lemma \ref{LEMMA_comparison_bis_preliminary}, we have that $(t_1^-,t_2^-) \subset (-T,T)$ with $t_2^- < T$. Therefore, since we assume that \ref{hyp_regularity} holds we have that $\bV_{c,T} \in \CA((-T,T))$ fulfills \eqref{bcT_equation} by Lemma \ref{LEMMA_bcT_solution}. The definition of $\CA((-T,T))$ in  \eqref{CA_def} implies that $t \in [t_1^-,t_2^-]  \to \CE(\bV_{c,T}(t))$ is continuous, which means that $\CE(\bV_{c,T}(t_2^-))=\alpha_0$. Due to the definitions of $t_1^-$ and $t_2^-$ in \eqref{comparison_t1} and \eqref{comparison_t2} respectively, we have that
\begin{equation}
\forall t \in (t_1^-,t_2^-), \hspace{2mm} \CE(\bV(t)) \in (\alpha_-,\alpha_0), \hspace{1mm} \bV(t) \in \CE^{\alpha_0,-}
\end{equation}
which means that we can use \eqref{omega_def} to obtain
\begin{equation}
\forall \bv^- \in \scrF^-, \forall t \in (t_1^-,t_2^-), \hspace{2mm} \frac{d}{d\lambda} (\CE(\bv^-+\lambda(\bV_{c,T}(t)-\bv^-))(1) \geq \omega > 0,
\end{equation}
which means, using \eqref{abstract_ipp} in \ref{hyp_smaller_space}
\begin{equation}\label{comparison_prop_ineq1}
\forall \bv^- \in \scrF^-, \forall t \in (t_1^-,t_2^-), \hspace{2mm} \langle D_{\scrL}(\bV_{c,T}),\bV_{c,T}(t)-\bv^- \rangle_{\scrL} \geq \omega.
\end{equation}
By \ref{hyp_levelsets_bis}, there exists $\bv^-_{c,T} \in \scrF^-$ such that $\lVert \bV_{c,T}(t_2^-)-\bv^-_{c,T} \rVert_{\scrL} \leq R$. Define now the function
\begin{equation}
\rho: t \in [t_1^-,t_2^-] \to \lVert \bV_{c,T}(t)-\bv^-_{c,T} \rVert_{\scrL}^2 \in \R
\end{equation}
so that
\begin{equation}\label{rho_t2}
\rho(t_2^-) \leq R.
\end{equation}
By \eqref{bcT_equation} and \ref{hyp_regularity}, we have that $\rho \in \CC^2((t_1,t_2))$. We have
\begin{equation}
\forall t \in (t_1^-,t_2^-), \hspace{2mm} \rho'(t) = \langle \bV_{c,T}'(t),\bV_{c,T}(t)-\bv^-_{c,T} \rangle_{\scrL}
\end{equation}
and, subsequently, using \eqref{bcT_equation}
\begin{equation}
\forall t \in (t_1^-,t_2^-), \hspace{2mm} \rho''(t)= \lVert \bV_{c,T} \rVert_{\scrL}^2+\langle D_\scrL(\bV_{c,T}(t)),\bV_{c,T}(t)-\bv^-_{c,T} \rangle_{\scrL}-c\rho'(t),
\end{equation}
which, by \eqref{comparison_prop_ineq1} means that
\begin{equation}\label{rho_ineq}
\forall t \in (t_1^-,t_2^-), \hspace{2mm} \rho''(t)+c\rho'(t) \geq \omega.
\end{equation}
Notice that \eqref{rho_ineq} implies that $\rho$ does not possess any local maximum in $[t_1^-,t_2^-]$, as we have $\omega>0$. That is, either
\begin{enumerate}
\item The function $\rho$ is non-decreasing in $[t_1^-,t_2^-]$.
\item The function $\rho$ possesses a unique local minimum at $t_3^- \in (t_1^-,t_2^-)$.
\end{enumerate}
Assume that 1. holds. Then for each $t \in (t_1^-,t_2^-)$ we can integrate \eqref{rho_ineq} in $[t_1^-,t]$ to obtain
\begin{equation}\label{int_t1t}
c(\rho(t)-\rho(t_1^-))+\rho'(t)-\rho'(t_1^-) \geq \omega (t-t_1^-)
\end{equation}
so that, since we assume that $\rho' \geq 0$ in $(t_1^-,t_2^-)$
\begin{equation}\label{rho_ineq_t1}
c\rho(t)+\rho'(t) \geq \omega(t-t_1^-)
\end{equation}
which, integrating now in $[t_1^-,t_2^-]$ gives
\begin{equation}
c\int_{t_1^-}^{t_2^-}\rho(t)dt + \rho(t_2^-) \geq \frac{\omega}{2}(t_2^--t_1^-)^2.
\end{equation}
By \eqref{rho_t2}, and since $\rho$ does not possess any local maximum, the previous becomes,
\begin{equation}\label{ineq_t1_t2}
Rc(t_2^--t_1^-)+R \geq \frac{\omega}{2}(t_2^--t_1^-)^2.
\end{equation}
The roots of the polynomial
\begin{equation}
\frac{\omega}{2}x^2-Rcx-R
\end{equation}
are
\begin{equation}\label{roots}
x_1= \frac{Rc + \sqrt{R^2c^2+2R\omega}}{\omega}>0 \mbox{ and } x_2= \frac{Rc-\sqrt{R^2c^2+2R\omega}}{\omega}<0,
\end{equation}
therefore, since \eqref{ineq_t1_t2} holds, we must have
\begin{equation}\label{ineq_t12_x1}
t_2^--t_1^- \leq x_1.
\end{equation}
Assume now that 1. does not hold. Then 2. holds. Notice that since $t_3^-$ is a local minimum of $\rho$ in $(t_1^-,t_2^-)$ we have $\rho'(t_3^-)=0$. We first integrate \eqref{rho_ineq} in $[t_3^-,t]$ for $t \in (t_3^-,t_2^-)$ to get
\begin{equation}
c(\rho(t)-\rho(t_3^-))+\rho'(t) \geq \omega(t-t_3^-)
\end{equation}
which means that
\begin{equation}
\forall t \in (t_3^-,t_2^-), \hspace{2mm} c\rho(t) + \rho'(t) \geq \omega(t-t_3^-),
\end{equation}
which is exactly \eqref{rho_ineq_t1} with $t_1^-$ replaced by $t_3^-$. The same reasoning as above gives
\begin{equation}\label{ineq_t13_x1}
t_2^--t_3^- \leq x_1
\end{equation}
with $x_1$ as in \eqref{roots}. We now take $t \in (t_1^-,t_3^-)$ and integrate \eqref{rho_ineq} in $[t,t_3^-]$. We obtain
\begin{equation}
-\rho'(t)+c(\rho(t_3^-)-\rho(t)) \geq \omega(t_3^--t)
\end{equation}
so that
\begin{equation}
\forall t \in [t_1^-,t_3^-], \hspace{2mm} -\rho'(t)+c\rho(t_3) \geq \omega(t_3^--t)
\end{equation}
which, by integrating in $[t_1^-,t_3^-]$ and using \eqref{rho_t2} along with the fact that $\rho$ has no local maximum in $[t_1^-,t_2^-]$, gives
\begin{equation}
R+Rc(t_3^--t_1^-) \geq \frac{\omega}{2}(t_3^--t_1)^2
\end{equation}
which is exactly \eqref{ineq_t1_t2} but with $t_2^-$ replaced by $t_3^-$. Therefore, we get
\begin{equation}
t_3^--t_1^- \leq x_1
\end{equation}
which in combination with \eqref{ineq_t13_x1} gives
\begin{equation}\label{ineq_t12_x1_2}
t_2^--t_1^- \leq 2x_1.
\end{equation}
As a consequence, we have that either \eqref{ineq_t12_x1} or \eqref{ineq_t12_x1_2} holds. That is, we have shown that \eqref{sfT1}, with $\sfT_1(c)$ as in \eqref{sfT1_definition}, holds. We now show that \eqref{comparison_bis_1_t1} holds. Otherwise, we would have $-T < t_1^-$ and some $\tilde{t} \in (-T,t_1^-)$ such that
\begin{equation}
\CE(\bV_{c,T}(\tilde{t})) > \alpha_-
\end{equation}
and, since we have that $\CE(\bV_{c,T}(-T)) $ and $\CE(\bV_{c,T}(t_1^-))$ are smaller than $\alpha_-$, we obtain a contradiction by Lemma \ref{LEMMA_comparison_bis_1}. Therefore, \eqref{comparison_bis_1_t1} has been established. We now show the existence $t_0^- \leq t_1^-$ and $\delta_0^->0$ such that \eqref{comparison_bis_1_t0} holds. We begin by showing the existence of $\tilde{t}_0^- \leq t_1^-$ such that
\begin{equation}\label{comparison_bis_1_t0_tilde}
\forall t \leq \tilde{t}_0^-, \hspace{2mm} \CE(\bV_{c,T}(t)) < \alpha_-.
\end{equation}
By contradiction, assume that there exists a sequence $(t_n)_{n \in \N}$ in $(-\infty,t_1^-]$ such that $t_n \to -\infty$ as $n \to \infty$ and
\begin{equation}\label{contradiction_alpha-}
\forall n \in \N, \hspace{2mm} \CE(\bV_{c,T}(t_n)) =  \alpha_-.
\end{equation}
Assume by contradiction that there exists a sequence $(\tilde{t}_n)_{n \in \N}$ in $(-\infty,t_1^-]$ such that $\tilde{t}_n \to -\infty$ and
\begin{equation}
\forall n \in \N, \hspace{2mm} \CE(\bV_{c,T}(\tilde{t}_n))< \alpha_-
\end{equation}
which, in combinantion with \eqref{contradiction_alpha-}, shows the existence of $t_a<t_b<t_c \leq t_1^-$ and $\tilde{\alpha}_-<\alpha_-<0$ such that
\begin{equation}
\bV_{c,T}(t_a), \bV_{c,T}(t_c) \in \CE^{\tilde{\alpha}_-}=\CE^{\tilde{\alpha}_-,-} \mbox{ and } \CE(\bV_{c,T}(t_b)) = \alpha_- > \tilde{\alpha}_-
\end{equation}
which, by Lemma \ref{LEMMA_comparison_bis_1}, gives a contradiction with the minimality of $\bV_{c,T}$. As a consequence, we have that there exists $\tilde{t} \leq t_1^-$ such that
\begin{equation}\label{contradiction_tilde_t}
\forall t \leq \tilde{t}, \hspace{2mm} \CE(\bV_{c,T}(t)) = \alpha_-.
\end{equation}
By Lemma \ref{LEMMA_comparison_bis_2}, we have that $\bV_{c,T}$ cannot be constant in any non-empty interval $I \subset (-\infty,\tilde{t}]$. Therefore, we have in particular that
\begin{equation}\label{energy_tilde_t}
\int_{-\infty}^{\tilde{t}}\frac{\lVert \bV_{c,T}'(t)\rVert_{\scrL}^2}{2}e^{ct}dt>0.
\end{equation}
Let $\tilde{U}$ be the function defined as
\begin{equation}
\tilde{U}(t)= \begin{cases}
\bV_{c,T}\left( \frac{1}{2}t+\frac{1}{2}\tilde{t} \right) &\mbox{ if } t \leq \tilde{t} \\
\bV_{c,T}(t) &\mbox{ if } t \geq \tilde{t}
\end{cases}
\end{equation}
which is well defined and belongs to $Y_T$. Notice that for all $t \leq \tilde{t}$, we have that $\frac{1}{2}t+\frac{1}{2}\tilde{t} \leq \tilde{t}$, so that $\CE(\tilde{U}(t))= \alpha_-$ by \eqref{contradiction_tilde_t}. Therefore,
\begin{equation}
\int_{-\infty}^{\tilde{t}} \CE(\tilde{U}(t)) e^{ct}dt= \int_{-\infty}^{\tilde{t}} \CE(\bV_{c,T}(t))e^{ct}dt,
\end{equation}
meaning that
\begin{equation}\label{energy_Utilde_1}
\int_\R \CE(\tilde{U}(t))e^{ct}dt = \int_\R \CE(\bV_{c,T}(t)) e^{ct}dt.
\end{equation}
Now, notice that for a. e. $t \leq \tilde{t}$ it holds $U'(t)=\frac{1}{2}\bV_{c,T}\left( \frac{1}{2}t+\frac{1}{2}\tilde{t} \right)$. Hence, we have that
\begin{align}
\int_{-\infty}^{\tilde{t}} \frac{\lVert \tilde{U}'(t) \rVert_{\scrL}^2}{2}e^{ct}dt&= \frac{1}{4} \int_{-\infty}^{\tilde{t}} \frac{\lVert \bV_{c,T}\left(  \frac{1}{2}t+\frac{1}{2}\tilde{t}\right) \rVert_{\scrL}^2}{2}e^{ct}dt\\ &\leq \frac{1}{4}\int_{-\infty}^{\tilde{t}}\frac{\lVert \bV_{c,T}\left(  \frac{1}{2}t+\frac{1}{2}\tilde{t}\right) \rVert_{\scrL}^2}{2}e^{c\left( \frac{1}{2}t+\frac{1}{2}\tilde{t} \right)}dt \\
&= \frac{1}{2} \int_{-\infty}^{\tilde{t}} \frac{\lVert \bV_{c,T}'(t)\rVert_{\scrL}^2}{2}e^{ct}dt,
\end{align}
which, by \eqref{energy_tilde_t}, means that
\begin{equation}
\int_\R \frac{\lVert \tilde{U}'(t) \rVert_{\scrL}^2}{2}e^{ct}dt < \int_\R \frac{\lVert \bV_{c,T}'(t)\rVert_{\scrL}^2}{2}e^{ct} dt,
\end{equation}
so that, taking also into account \eqref{energy_Utilde_1}, we have obtained
\begin{equation}
\bE_c(\tilde{U})< \bE_c(\bV_{c,T}),
\end{equation}
which is a contradiction, because $\bV_{c,T}$ minimizes $\bE_c$ in $Y_T$. As a consequence, \eqref{contradiction_alpha-} cannot hold, which means that there exists $\tilde{t}_0^- \leq t_1^-$ such that \eqref{comparison_bis_1_t0_tilde} holds. We now establish the existence of $t_0^- \leq t_1^-$ and $\delta_0^->0$ such that \eqref{comparison_bis_1_t0} holds. Applying Lemma \ref{LEMMA_comparison_bis_1}, we have that $\CE(\bV_{c,T}(\cdot))$ does not possess any strict local maximum, and at most one strict local minimum, in $(-\infty,t_1^-]$. Therefore, there exist $t_0^- \leq \tilde{t}_0^-$ and $\delta_0^->0$ such that $\CE(\bV_{c,T}(\cdot))$ is monotone in $(-\infty,t_0^-+\delta_0^-]$ and $t_0^-+\delta_0^- \leq \tilde{t}_0^-$. Therefore, \eqref{comparison_bis_1_t0} follows by \eqref{comparison_bis_1_t0_tilde}.
\end{proof}
For $U \in Y_T$, recall the definition of $t^+(U,\eps_0^+)$ in \eqref{comparison_t+}. The goal now will be to provide a uniform bound on $t^+(U,\eps_0^+)-t_2^-(U)$. The result is very close to the second part of the comparison result in \cite{oliver-bonafoux-tw}.
\begin{proposition}\label{PROPOSITION_comparison_bis_+}
Assume that \ref{hyp_compactness}, \ref{hyp_projections_inverse}, \ref{hyp_projections_H}, \ref{hyp_regularity} and \ref{hyp_levelsets_bis} hold. Let $c > 0$ and $T \geq 1$. Consider $\bV_{c,T} \in Y_T$ a minimizer of $\bE_c$ in $Y_T$, which exists by Lemma \ref{LEMMA_bcT}. Let $t_2^-:= t_2^-(\bV_{c,T})$ as in \eqref{comparison_t2} and $t^+:= t^+(\bV_{c,T},\eps_0^+)$ as in \eqref{comparison_t+}. Then, we have that
\begin{equation}\label{comparison_bis_+_t+}
\forall t \geq t^+, \hspace{2mm} \dist_{\scrL}(\bV_{c,T}(t),\scrF^+) < \frac{r_0^+}{2}
\end{equation}
Furthermore, we have that
\begin{equation}\label{sfT_2}
0 < t^+-t_2^- \leq \sfT_2(c),
\end{equation}
where
\begin{equation}\label{sfT_2_definition}
\sfT_2(c):= \frac{1}{c} \ln\left( \frac{-a}{\alpha_{\sstar}}+1 \right)
\end{equation}
with $\alpha_{\sstar}>0$ a constant independent on $c$, $T$ and $U$.
\end{proposition}
\begin{proof}
We claim that for all $t \geq t^+$, it holds that $\CE(\bV_{c,T}(t)) \geq 0$. Indeed, otherwise we could find $\tilde{t} \geq t^+$ such that $\CE(\bV_{c,T}(\tilde{t}))=\alpha < 0$.  Since we necessarily have $t^+>-T$, and $\bV_{c,T}(-T) \in \CE^{\max\{\alpha_-,\alpha\}}$, by Lemma \ref{LEMMA_comparison_bis_1}, we can find $\tilde{U} \in Y_T$ such that $\bE_c(\tilde{U}) <  \bE_c(\bV_{c,T})$, a contradiction. Therefore, arguing as in the second part of the comparison result in \cite{oliver-bonafoux-tw} we obtain that \eqref{comparison_bis_+_t+} must hold. In order to prove \eqref{sfT_2}, we first show that $t^+ > t_2^-$. This follows from the fact that for all $t \leq t_2^-$ we must have $\bV_{c,T}(t) \in \CE^{\alpha_0,-}$, $\bV_{c,T}(t^+) \in \scrF^+_{r_0^+/2}$ and $\scrF^+_{r_0^+/2} \cap \CE^{\alpha_0,-}=\emptyset$. See \ref{hyp_levelsets_bis}. Let $\alpha_{\sstar}:=\min\{\alpha_0,\eps_0^+\}$. We claim that
\begin{equation}\label{alpha_starstar}
\forall t \in [t_2^-,t^+], \hspace{2mm} \CE(\bV_{c,T}(t)) \geq \alpha_{\sstar}.
\end{equation}
Indeed, assume by contradiction that for some $t \in (t_2^-,t^+)$ we have $\CE(\bV_{c,T}(t)) = \alpha < \alpha_{\sstar}$. If $\bV_{c,T}(t) \in \scrF^+_{r_0^+/2}$, we have a contradiction with the definition of $t^+$ since $\alpha< \eps_0^+$. Assume then $\bV_{c,T}(t) \not \in \scrF^+_{r_0^+/2}$. Since $\CE^{\alpha,+} \subset \CE^{\alpha_0,+} \subset \scrF^+_{r_0^+/2}$ by \ref{hyp_levelsets_bis}, we have that $\bV_{c,T}(t) \in \CE^{\alpha,-}$. We have that $\bV_{c,T}(-T) \in \CE^{\alpha_-}\subset \CE^{\alpha,-}$ and $\bV_{c,T}(t_2^-) = \alpha_0 > \alpha$ by Proposition \ref{PROPOSITION_comparison_bis_1}. Therefore, Lemma \ref{LEMMA_comparison_bis_1} leads to a contradiction. As a consequence, we have shown that \eqref{alpha_starstar} holds. Arguing as in the last part of the corresponding result in \cite{oliver-bonafoux-tw}, we obtain that \eqref{sfT_2}, with $\sfT_2(c)$ as in \eqref{sfT_2_definition} holds. 
\end{proof}
Propositions \ref{PROPOSITION_comparison_bis_1} and \ref{PROPOSITION_comparison_bis_+} can be summarized as follows:
\begin{corollary}\label{COROLLARY_comparison_bis}
Assume that \ref{hyp_compactness}, \ref{hyp_projections_inverse}, \ref{hyp_projections_H}, \ref{hyp_regularity} and \ref{hyp_levelsets_bis} hold. Let $c > 0$ and $T \geq 1$. Consider $\bV_{c,T} \in Y_T$ a minimizer of $\bE_c$ in $Y_T$, which exists by Lemma \ref{LEMMA_bcT}. Let $t_1^-:=t_1^-(\bV_{c,T})$ as in \eqref{comparison_t1} and $t^+:= t^+(\bV_{c,T},\eps_0^+)$ as in \eqref{comparison_t+}. Then it holds that
\begin{equation}\label{corollary_comparison_bis+}
\forall t \geq t^+, \hspace{2mm} \dist_{\scrL}(\bV_{c,T}(t),\scrF^+) < \frac{r_0^+}{2},
\end{equation}
\begin{equation}\label{corollary_comparison_bis_t1-}
\forall t \leq t_1^-, \hspace{2mm} \CE(\bV_{c,T}(t)) \leq \alpha_-
\end{equation}
and there exist $\delta_0^->0$ and $t_0^- \leq t_1^-$ such that
\begin{equation}\label{corollary_comparison_bis_t0-}
\begin{cases}
\CE(\bV_{c,T}(\cdot)) \mbox{ is monotone in } (-\infty,t_0^-+\delta_0^-] \mbox{ and }\\
\forall t \leq t_0^-+\delta_0^-, \hspace{2mm} \CE(\bV_{c,T}(t)) < \alpha_-.
\end{cases}
\end{equation}
Moreover, we have
\begin{equation}\label{sfT_starstar}
0<t^+-t_1^- \leq \sfT_{\sstar}(c)
\end{equation}
where
\begin{equation}\label{sfT_starstar_definition}
\sfT_{\sstar}(c):=\sfT_1(c)+\sfT_2(c)
\end{equation}
with $\sfT_1(c)$ as in \eqref{sfT1_definition} and $\sfT_2(c)$ as in \eqref{sfT_2_definition}. That is,
\begin{equation}
c \in (0,+\infty) \to \sfT_{\sstar}(c) \in (0,+\infty)
\end{equation}
is continuous and independent on $T$. Finally, if $t_2^-:=t_2^-(\bV_{c,T})$ as in \eqref{comparison_t2}, we have $t_2^- \in (t_1^-,t^+)$, $\bV_{c,T}(t_2^-) \in \CE^{\alpha_0,-}$ and
\begin{equation}\label{positivity_bis}
\forall t \geq t_2^-, \hspace{2mm} \CE(\bV_{c,T}(t)) \geq 0.
\end{equation}
\end{corollary}

Finally, we have the following property on constrained solutions:
\begin{corollary}\label{COROLLARY_bcT_solution}
Assume that \ref{hyp_compactness}, \ref{hyp_projections_inverse}, \ref{hyp_projections_H}, \ref{hyp_regularity} and \ref{hyp_levelsets_bis} hold. Let $c > 0$ and $T \geq 1$. Consider $\bV_{c,T} \in Y_T$ a minimizer of $\bE_c$ in $Y_T$, which exists by Lemma \ref{LEMMA_bcT}. Let $t_1^-:=t_1^-(\bV_{c,T})$ as in \eqref{comparison_t1} and $t^+:= t^+(\bV_{c,T},\eps_0^+)$ as in \eqref{comparison_t+} and $t_0^- \leq t_1^-$, also given by Corollary \ref{COROLLARY_comparison_bis}. Then, there exists $\delta_{Y,c,T}>0$ such that the set
\begin{equation}\label{SYcT}
S_{Y,c,T}:=(-\infty,t_0^-+\delta_{Y,c,T}) \cup (-T,T) \cup (t^+-\delta_{Y,c,T},+\infty)
\end{equation}
is such that $\bV_{c,T} \in \CA(S_{Y,c,T})$ (see \eqref{CA_def}) and
\begin{equation}\label{corollary_bcT_solution_equation}
\bV_{c,T}''-D_\scrL\CE(\bV_{c,T})=-c\bV_{c,T}' \mbox{ in } S_{Y,c,T}.
\end{equation}
\end{corollary}
\begin{proof}
Using \eqref{corollary_comparison_bis_t0-} in Corollary \ref{COROLLARY_comparison_bis}, we have that for each $\tilde{T}>0$ it holds
\begin{equation}
\forall t \leq [t_0^--\tilde{T},t_0^-+\delta_0^-], \hspace{2mm} \CE(\bV_{c,T}(t)) \leq \tilde{\alpha} < \alpha_-,
\end{equation}
so that \eqref{bcT_no_constrain_-} in Lemma \ref{LEMMA_bcT_solution} is satisfied for any closed interval contained in  $(-\infty,t_0^-]$. As a consequence, \eqref{bcT_equation_delta-} follows in such interval. That means that for $t_0^- \leq t_1^-$ and $\delta_0^->0$ as in Corollary \ref{COROLLARY_comparison_bis}, we have that $\bV_{c,T} \in \CA((-\infty,t_0^-+\delta_0^-))$ and
\begin{equation}
\bV_{c,T}''-D_\scrL\CE(\bV_{c,T})=-c\bV_{c,T}'' \mbox{ in } (-\infty,t_0^-+\delta_0^-) .
\end{equation}
Using again Lemma \ref{LEMMA_bcT_solution} in combination with \eqref{corollary_comparison_bis+} in Corollary \ref{COROLLARY_comparison_bis}, we obtain for some $\delta_0^+>0$ that
\begin{equation}
\bV_{c,T} \in \CA((-T,T) \cup (t^+-\delta_{0}^+,+\infty))
\end{equation}
and
\begin{equation}
\bV_{c,T}''-D_\scrL\CE(\bV_{c,T})=-c\bV_{c,T}'' \mbox{ in }  (-T,T) \cup  (t^+-\delta_{0}^+,+\infty) .
\end{equation}
As a consequence, \eqref{corollary_bcT_solution_equation} follows by taking $\delta_{Y,c,T}:=\min\{\delta_0^-,\delta_0^+\}>0$.
\end{proof}
\subsection{Existence of an unconstrained solution}
We can now establish the existence of an unconstrained solution for the proper speed. Let
\begin{equation}\label{set_CD}
\CD:= \{ c>0: \exists T \geq 1 \mbox{ and } U \in Y_T \mbox{ such that } \bE_c(U) < 0\}.
\end{equation}
We have:
\begin{lemma}\label{LEMMA_set_CD}
Assume that \ref{hyp_compactness}, \ref{hyp_projections_inverse}, \ref{hyp_projections_H} and \ref{hyp_levelsets_bis} hold. Let $\CD$ be the set defined in \eqref{set_CD}.  Then, $\CD$ is open and non-empty and bounded such that
\begin{equation}\label{sup_D_bound}
\sup \CD \leq \frac{\sqrt{-2a}}{d_{\alpha_0}},
\end{equation}
where $d_{\alpha_0}>0$ is as in \eqref{d_alpha_0}
\end{lemma}
\begin{proof}
Let $\Psi$ be the function defined in \eqref{Psi}. Since \ref{hyp_levelsets_bis} holds we have by Lemma \ref{LEMMA_PSI_YT} that $\Psi \in Y_T$. Therefore, the argument from \cite{oliver-bonafoux-tw} applies here to show that $\CD \not = \emptyset$. Subsequently, let $c \in \CD$. By definition, there exists $T \geq 1$ such that $\bE_c(\bV_{c,T})<0$, where $\bV_{c,T} \in Y_T$ is a minimizer of $\bE_c$ in $Y_T$, which exists by Lemma \ref{LEMMA_bcT}. Let $\bv_{c,T}^+$ be such that $\lVert \bV_{c,T}(t)-\bv_{c,T}^+ \rVert_{\scrH} \to 0$ as $t \to +\infty$, which occurs due to Lemma \ref{LEMMA_limit+}. Define for $t \geq T$
\begin{equation}
\bV_{c,T}^t(s):= \begin{cases}
\bV_{c,T}(s) &\mbox{ if } s \leq t,\\
(1+t-s)\bV_{c,T}(t)+(s-t)\bv_{c,T}^+ &\mbox{ if } t \leq s \leq t+1,\\
\bv_{c,T}^+ &\mbox{ if } t+1 \leq s.
\end{cases}
\end{equation}
which belongs to $Y_T$. Arguing as in \cite{oliver-bonafoux-tw}, we can show that for $t \geq T$ large enough and some $\delta>0$ we have for all $\tilde{c} \in (c-\delta,c+\delta)$ that $\bE_{\tilde{c}}(\bV_{c,T}^t(s))<0$. This shows that $\CD$ is open. In order to establish the bound \eqref{sup_D_bound}, we use Corollary \ref{COROLLARY_comparison_bis}. Let $c \in \CD$, there exists then $T \geq 1$ such that $\bE_c(\bV_{c,T})<0$, with $\bV_{c,T}$ a minimizer of $\bE_c$ in $Y_T$. Let $t_2^-:=t_2^-(\bV_{c,T})$ as in \eqref{comparison_t2} and $t^+:=t^+(\bV_{c,T},\eps_0^+)$ as in \eqref{comparison_t+}. Corollary \ref{COROLLARY_comparison_bis} implies that $\bV_{c,T}(t_2^-) \in \CE^{\alpha_0-}$ and $\bV_{c,T}(t^+) \in \scrF^+_{r_0^+/2}$. Therefore, the definition of $d_{\alpha_0}$ in \eqref{d_alpha_0} implies
\begin{equation}
\left( d_{\alpha_0} \right)^2\leq \lVert \bV_{c,T}(t^+)-\bV_{c,T}(t_2^-) \rVert_{\scrL}^2
\end{equation}
so that
\begin{equation}
\left( d_{\alpha_0} \right)^2\leq 2 \int_{\R} \frac{\lVert \bV_{c,T}'(t) \rVert_{\scrL}^2}{2} e^{ct}dt \left( \frac{e^{-ct_2^-}-e^{-ct^+}}{c} \right).
\end{equation}
Using \eqref{positivity_bis} in Corollary \ref{COROLLARY_comparison_bis}, the inequality above becomes
\begin{equation}
\left( d_{\alpha_0} \right)^2\leq 2\left( \bE_c(\bV_{c,T})-\frac{e^{ct_2^-}}{c} \right) \left( \frac{e^{-ct_2^-}-e^{-ct^+}}{c} \right)
\end{equation}
which, using that $ \bE_c(\bV_{c,T})<0$ and $t^+ > t_2^-$, gives \eqref{sup_D_bound} and concludes the proof.
\end{proof}
We can now establish the existence of an unconstrained solution:
\begin{proposition}\label{PROPOSITION_existence_bis}
Assume that \ref{hyp_compactness}, \ref{hyp_projections_inverse}, \ref{hyp_projections_H}, \ref{hyp_regularity} and \ref{hyp_levelsets_bis} hold.  Let $\overline{c} \in \partial(\CD) \cap (0,+\infty)$, where $\partial(\CD)$ stands for the boundary of the set $\CD$ defined in \eqref{set_CD}. Then, there exists $\overline{T} \geq 1$ such that $\bb_{\overline{c},\overline{T}}=0$ ($\bb_{\overline{c},\overline{T}}$ as in \eqref{bcT}) and $\overline{\bV} \in Y_{\overline{T}}$ an associated minimizer of $\bE_{\overline{c}}$ which does not saturate the constraints, i. e.,
\begin{equation}\label{unconstrained+_bis}
\forall t \geq \overline{T}, \hspace{2mm} \dist_{\scrL}(\overline{\bV}(t),\scrF^+) < \frac{r_0^+}{2}
\end{equation}
and
\begin{equation}\label{unconstrained-_bis}
\forall t \leq -\overline{T}, \hspace{2mm} \CE(\overline{\bV}(t)) < \alpha_-.
\end{equation}
In particular, the pair $(\overline{c},\overline{\bV})$ solves \eqref{abstract_equation}.
\end{proposition}
\begin{remark}
Notice that Lemma \ref{LEMMA_set_CD} implies that (under the necessary assumptions) the set $\CD$ is bounded, meaning that the set $\partial(\CD) \cap (0,+\infty)$ is not empty. Such a fact, in combination with Proposition \ref{PROPOSITION_existence_bis}, shows the existence of the unconstrained solution.
\end{remark}
\begin{proof}[Proof of Proposition \ref{PROPOSITION_existence_bis}]
We essentially mimic the proof from \cite{oliver-bonafoux-tw} with the obvious modifications.  Lemma \ref{LEMMA_set_CD} implies that $\partial(\CD) \subset \R \setminus \CD$, so that $\overline{c} \not \in \CD$. The definition of $\CD$ in \eqref{set_CD} implies then that
\begin{equation}\label{c^sstar_nonnegative}
\forall T \geq 1, \hspace{2mm} \bb_{\overline{c},T} \geq 0.
\end{equation}
Let $(c_n)_{n \in \N}$ be a sequence in $\CD$ such that $c_n \to \overline{c}$ as $n \to \infty$. By definition, for each $n \in \N$ there exists $T_n \geq 1$ such that $\bE_{c_n}(\bV_{c_n,T_n})<0$ where, for each $n \in \N$, $\bV_{c_n,T_n}$ is a minimizer of $\bE_{c_n}$ in $Y_{T_n}$. For each $n \in \N$, set $t_{1,n}^-:=t_1^-(\bV_{c_n,T_n})$ as in \eqref{comparison_t1} and $t_n^+:=t^+(\bV_{c_n,T_n},\eps_0^+)$ as in \eqref{comparison_t+}. Using \eqref{sfT_starstar} in Corollary \ref{COROLLARY_comparison_bis} we have
\begin{equation}
\forall n \in \N, \hspace{2mm} 0 < t_n^+ - t_{1,n}^+ \leq \sfT_{\sstar}(c_n)
\end{equation}
with
\begin{equation}
c \in (0,+\infty) \to \sfT_{\sstar}(c) \in (0,+\infty)
\end{equation}
a continuous function. Since the sequence $(c_n)_{n \in \N}$ is bounded, we have that
\begin{equation}
\sfT_\sstar:= \max\left\{1,\sup_{n \in \N} \sfT_\sstar(c_n)\right\}<+\infty
\end{equation}
and
\begin{equation}\label{unconstrainted_sfT_bis}
\forall n \in \N, \hspace{2mm}0 < t^+_n-t^-_{1,n} \leq \sfT_\sstar,
\end{equation}
so that we have a bound on $(t^+_n-t^-_{1,n})_{n \in \N}$. Moreover, \eqref{corollary_comparison_bis+} and \eqref{corollary_comparison_bis_t1-} in Corollary \ref{COROLLARY_comparison_bis} imply
\begin{equation}\label{unconstrained_t+_bis}
\forall n \in \N, \forall t \geq t^+_{1,n}, \hspace{2mm} \dist_{\scrL}(\bV_{c_n,T_n}(t),\scrF^+) < \frac{r_0^+}{2}
\end{equation}
and
\begin{equation}\label{unconstrained_t-_bis}
\forall n \in \N, \forall t \leq t^-_n, \hspace{2mm} \CE(\bV_{c_n,T_n}(t)) \leq \alpha_-.
\end{equation}
For each $n \in \N$, define the function $\bV_{c_n,T_n}^{t_n^+}:=\bV_{c_n,T_n}(\cdot + t_n^+)$. Then, \eqref{unconstrainted_sfT_bis} implies that \eqref{unconstrained_t+_bis} and \eqref{unconstrained_t-_bis} write as
\begin{equation}\label{translate_unconstrained_+_bis}
\forall n \in \N, \forall t \geq 0, \hspace{2mm} \dist_{\scrL}(\bV_{c_n,T_n}^{t_n^+}(t),\scrF^-) < \frac{r_0^-}{2}.
\end{equation}
and
\begin{equation}\label{translate_unconstrained_t-_bis}
\forall n \in \N, \forall t \leq -\sfT_\sstar, \hspace{2mm} \CE(\bV_{c_n,T_n}^{t_n^+}(t)) \leq \alpha_-
\end{equation}
so that for all $n \in \N$ we have $\bV_{c_n,T_n}^{t_n^+} \in Y_{\sfT_\sstar}$. Moreover, a computation shows
\begin{equation}
\forall n \in \N, \hspace{2mm} \bE_{c_n}(\bV_{c_n,T_n}^{t_n^+})=e^{-c_nt_n^+}\bE_{c_n}(\bV_{c_n,T_n})<0.
\end{equation}
Therefore, if we apply Lemma \ref{LEMMA_LSC_YT} with sequence of speeds $(c_n)_{n \in \N}$ and the sequence $(\bV_{c_n,T_n}^{t_n^+})_{n \in \N}$ in $Y_{\sfT_\sstar}$, we obtain $\overline{\bV} \in Y_{\sfT_\sstar}$ such that
\begin{equation}
\bE_{\overline{c}}(\overline{\bV}) \leq \liminf_{n \to \infty } \bE_{c_n}(\bV_{c_n,T_n}^{t_n^+}) \leq 0,
\end{equation}
which in combination with \eqref{c^sstar_nonnegative} implies that $\bb_{\overline{c},\sfT_\sstar}=0$. Therefore, we have $\bE_{\overline{c}}(\overline{\bV})=0$, so that $\overline{\bV}$ is a minimizer of $\bE_{\overline{c}}$ in $Y_{\sfT_\sstar}$. Set $t^-_{1,\sstar}:= t_1^-(\overline{\bV})$ as in \eqref{comparison_t1} and $t^+_{\sstar}:=t^+(\overline{\bV},\eps_0^\pm)$ as in \eqref{comparison_t+}. Invoking \eqref{corollary_comparison_bis+} and \eqref{corollary_comparison_bis_t0-} in Corollary \ref{COROLLARY_comparison_bis}, we obtain as before that for some $t_0^- \leq t_1^-$ and $\delta_0^->0$
\begin{equation}\label{t_star+_bis}
\forall t \geq t^+_\sstar, \hspace{2mm} \dist_{\scrL}(\overline{\bV}(t),\scrF^+) < \frac{r_0^+}{2},
\end{equation} 
\begin{equation}\label{t_star0_bis}
\begin{cases}
\CE(\overline{\bV}(\cdot)) \mbox{ is monotone in } (-\infty,t_0^-+\delta_0-] \mbox{ and }\\
\forall t \leq t_0^-+\delta_0^-, \hspace{2mm} \CE(\overline{\bV}(t)) < \alpha_-.
\end{cases}
\end{equation}
If we set $\overline{T}=\max\{1,t^+_\sstar,-t_0^-\}$, then \eqref{t_star+_bis} and \eqref{t_star0_bis} imply that $\overline{\bV} \in Y_{\overline{T}}$ and that \eqref{unconstrained+_bis}, \eqref{unconstrained-_bis} hold. Moreover, we have that $\bE_{\overline{c}}(\overline{\bV})=0$, so that $\overline{\bV}$ is a minimizer of $\bE_{\overline{c}}$ in $Y_{\overline{T}}$ due to \eqref{c^sstar_nonnegative}. We will now apply Corollary \ref{COROLLARY_bcT_solution} with $\overline{c}$, $\overline{T}$ as constants and $\overline{\bV}$ as constrained minimizer. By \eqref{t_star+_bis} and \eqref{t_star0_bis}, we have that $\overline{\bV} \in \CA(S_{Y,\overline{c},\overline{T}})$ and
\begin{equation}
\overline{\bV}''-D_\scrL\CE(\overline{\bV})=-\overline{c}\overline{\bV}' \mbox{ in } S_{Y,\overline{c},\overline{T}},
\end{equation}
with
\begin{equation}
S_{Y,\overline{c},\overline{T}}:=(-\infty,t_0^-+\delta_{Y,\overline{c},\overline{T}}) \cup (-\overline{T},\overline{T}) \cup (t_\sstar^+-\delta_{Y,\overline{c},\overline{T}},+\infty),
\end{equation}
for some $\delta_{Y,\overline{c},\overline{T}}>0$. The choice of $\overline{T}$ implies that $S_{Y,\overline{c},\overline{T}}=\R$, which concludes the proof.
\end{proof}
\subsection{Uniqueness of the speed}
We have the following:
\begin{proposition}[\cite{oliver-bonafoux-tw}]\label{PROPOSITION-uniqueness}
Assume that \ref{hyp_regularity} holds. Let $X$ be the set defined in \eqref{X}. Let $(c_1,c_2) \in (0,+\infty)^2$ be such that there exist $\bU_1$ and $\bU_2$ in $X\cap \CA(\R)$ such that $(c_1,\bU_1)$ and $(c_2,\bU_2)$ solve \eqref{abstract_equation} and for each $i \in \{1,2\}$, $\bE_{c_i}(\bU_i) <+\infty$. Assume moreover that
\begin{equation}\label{uniqueness_condition}
\forall i \in \{1,2\},\forall j \in \{1,2\} \setminus \{i\}, \hspace{2mm} \bE_{c_i}(\bU_j) \geq 0.
\end{equation}
Then, we have $c_1=c_2$.
\end{proposition}
This allows to show the following:

\begin{corollary}\label{COROLLARY_CD}
Assume that \ref{hyp_compactness}, \ref{hyp_projections_inverse}, \ref{hyp_projections_H}, \ref{hyp_regularity} and \ref{hyp_levelsets_bis} hold. Let
\begin{equation}\label{c(CD)}
c(\CD):= \sup \CD.
\end{equation}
Then, we have $\CD=(0,c(\CD))$.
\end{corollary}
\begin{proof}[Proof of Corollary \ref{COROLLARY_CD}]
The idea is the same than in \cite{oliver-bonafoux-tw}. As before, we need to show that
\begin{equation}
\partial(\CD)\cap(0,+\infty)=\{ c(\CD)\}.
\end{equation}
Let $\overline{c} \in \partial(\CD)\cap(0,+\infty)$, we need to show that $\overline{c}=c(\CD)$ in order to finish the proof. Using Proposition \ref{PROPOSITION_existence_bis} we obtain $\bV^\CD$ and $\overline{\bV}$ in $Y \subset X$ such that $(c(\CD),\bV(\CD))$ and $(\overline{c},\overline{\bV})$ solve \eqref{abstract_equation}. Moreover,
\begin{equation}
\inf_{U \in Y}\bE_{c(\CD)}(U)=\bE_{c(\CD)}(\bV^\CD)=0=\bE_{\overline{c}}(\overline{\bV})=\inf_{U \in Y}\bE_{\overline{c}}(U)
\end{equation}
so that we also have
\begin{equation}
\bE_{c(\CD)}(\overline{\bV}) \geq 0 \mbox{ and } \bE_{\overline{c}}(\bV^\CD) \geq 0.
\end{equation}
Therefore, the requirements of Proposition \ref{PROPOSITION-uniqueness} are met, which means that $\overline{c}=c(\CD)$.
\end{proof}

\subsection{Proof of Theorem \ref{THEOREM-ABSTRACT} completed}

If  \ref{hyp_levelsets_bis} holds, using Proposition \ref{PROPOSITION_existence_bis} and Corollary \ref{COROLLARY_CD}, we can argue as in \cite{oliver-bonafoux-tw} to establish the existence of a solution solving \eqref{abstract_equation} with condition \eqref{abstract_bc_weak} and speed $c^\star=c(\CD)$. The statement regarding the uniqueness of the speed $c^\star$ follows from Proposition \ref{PROPOSITION-uniqueness}. Finally, we have that the exponential convergence \eqref{exponential_abstract} follows from \eqref{limit+_function0_weak} in Lemma \ref{LEMMA_limit+}. Thus, the proof of Theorem \ref{THEOREM-ABSTRACT} is completed.

\qed

\subsection{Asymptotic behavior of the constrained solutions}
We essentially recall the results given in \cite{oliver-bonafoux-tw} and adapt them to this setting. Recall first the following:
\begin{lemma}[\cite{oliver-bonafoux-tw}]\label{LEMMA_equipartition}
Assume that \ref{hyp_regularity} holds. Let $c>0$, $t_1 < t_2$ and $U \in \CA((t_1,t_2))$ such that
\begin{equation}\label{equation_equipartition}
U''-D_{\scrL}\CE(U)=-cU' \mbox{ in }(t_1,t_2).
\end{equation}
Then, we have the formula
\begin{equation}\label{equipartition_identity}
\forall t \in (t_1,t_2), \hspace{2mm} \frac{d}{dt}\left( \CE(U(t))-\frac{\lVert U'(t) \rVert_{\scrL}^2}{2} \right) = c\lVert U'(t) \rVert^2_{\scrL}.
\end{equation}
\end{lemma}
We have the following result:
\begin{lemma}\label{LEMMA_equipartition_pointwise}
Assume that \ref{hyp_compactness}, \ref{hyp_projections_inverse}, \ref{hyp_projections_H}, \ref{hyp_regularity} and \ref{hyp_levelsets_bis} hold. Let $\bU_{c,T}$ be a constrained solution given by Lemma \ref{LEMMA_bcT} and $t^-:= t^-(\bU_{c,T},\CE_{\max}^-)$ be as in \eqref{comparison_t-}. Then for all $t<t^-$ we have the inequality
\begin{equation}\label{equipartition-_inequality}
\frac{\lVert \bU_{c,T}'(t) \rVert_{\scrL}^2}{2} \leq \CE(\bU_{c,T}(t))-a.
\end{equation}
Similarly, it holds that for all $t>t^+$
\begin{equation}\label{equipartition+_inequality}
\CE(\bU_{c,T}(t)) \leq \frac{\lVert \bU_{c,T}'(t) \rVert_{\scrL}^2}{2},
\end{equation}
where $t^+:=t^+(\bU_{c,T},\CE_{\max}^+)$ is as in \eqref{comparison_t+}.
\end{lemma}
The proof goes exactly as \textit{Lemma 5.13} in \cite{oliver-bonafoux-tw}, with the obvious minor modifications. From Lemma \ref{LEMMA_equipartition_pointwise} it follows a convergence result at $-\infty$, which one also might prove by modifying from \cite{oliver-bonafoux-tw}:
\begin{proposition}\label{PROPOSITION_conv_sol_-}
Assume that \ref{hyp_compactness}, \ref{hyp_projections_inverse}, \ref{hyp_projections_H}, \ref{hyp_regularity} and \ref{hyp_levelsets_bis} hold. Let $c>0$ and $T \geq 1$. Assume moreover that $c<\gamma^-$, where $\gamma^-$ is defined in \eqref{gamma-}. Let $\bU_{c,T}$ be a constrained solution given by Lemma \ref{LEMMA_bcT}.  Then, there exists $\overline{M}^->0$  such that for all $\eps \in (0,\gamma^--c)$ and $t \in \R$ it holds
\begin{equation}\label{L1_CE_conv-}
\int_{-\infty}^{t}(\CE(\bU_{c,T}(s ))-a)e^{-\eps s}ds \leq \overline{M}^- e^{(\gamma^--c-\eps)t}.
\end{equation}
Furthermore, there exist $M^->0$ and $\bv_{c,T}^-\in \scrF^-$ such that for all $t \in \R$
\begin{equation}\label{solutions_limit-_norm}
\lVert \bU_{c,T}(t)-\bv_{c,T}^- \rVert_{\scrL}^2 \leq M^- e^{(\gamma^--c)t}.
\end{equation}
\end{proposition}
In order to close this section, we prove the following result (which in fact generalizes Theorem \ref{THEOREM_abstract_isolated}) for the case in which \ref{hyp_levelsets_bis} holds and $\scrF^-$ is a singleton:

\begin{lemma}\label{LEMMA_isolated}
Assume that \ref{hyp_compactness}, \ref{hyp_projections_inverse}, \ref{hyp_projections_H}, \ref{hyp_regularity} and \ref{hyp_levelsets_bis} hold. Let $c>0$ and $T \geq 1$. Assume moreover that $\scrF^-=\{\bv^-\}$. Let $\bV_{c,T}$ be a constrained solution given by Lemma \ref{LEMMA_bcT_solution}. Then, it holds
\begin{equation}
\lim_{t \to -\infty} \lVert \bV_{c,T}(t) -\bv^- \rVert_{\scrL} =0.
\end{equation}
\end{lemma}
\begin{proof}
We essentially argue similarly to \cite{alikakos-fusco-smyrnelis} and \cite{alikakos-katzourakis}, using the fact that $\scrF^-=\{\bv^-\}$. In such a case, \eqref{local_monotonicity} in assumption \ref{hyp_levelsets_bis} implies that
\begin{equation}\label{isolated_local_monotonicity_1}
\forall v \in \CE^{\alpha_0,-}, \exists \lambda(v) >1 , \forall \theta \in (0,\lambda(v)), \hspace{2mm} \frac{d}{d\lambda}(\CE(\bv^-+\lambda(v-\bv^-)) \geq 0.
\end{equation}
Corollary \ref{COROLLARY_bcT_solution} implies in particular the existence of $t_0^- \in \R$ such that
\begin{equation}\label{isolated_alpha-}
\forall t \leq t_0^-, \hspace{2mm} \CE(\bV_{c,T}(t)) < \alpha_-,
\end{equation}
\begin{equation}
\bV \in \CA((-\infty,t_0^-))
\end{equation}
and
\begin{equation}\label{isolated_equation}
\bV''_{c,T}-D_\scrL\CE(\bV_{c,T})=-c\bV_{c,T}' \mbox{ in } (-\infty,t_0^-).
\end{equation}
The combination of \eqref{isolated_local_monotonicity_1} and \eqref{isolated_alpha-} gives that
\begin{equation}
\forall t \leq t_0^-, \exists \lambda(t)>1, \forall \theta \in (0,\lambda(t)), \hspace{2mm} \frac{d}{d\lambda}(\CE(\bv^-+\lambda(\bV_{c,T}(t)-\bv^-)) \geq 0
\end{equation}
which, using that $\bV_{c,T} \in \CA((-\infty,t_0))$ and \eqref{abstract_ipp} in \ref{hyp_smaller_space}, implies
\begin{equation}\label{isolated_omega}
\forall t < t_0^-, \hspace{2mm} \langle D_{\scrL}\CE(\bV_{c,T}(t)), \bV_{c,T}(t)-\bv^- \rangle_{\scrL} \geq 0.
\end{equation}
Define now the function
\begin{equation}
\rho_{c,T}: t \in (-\infty,t_0^-) \to \lVert \bV_{c,T}(t)-\bv^- \rVert_{\scrL}^2 \in \R.
\end{equation}
Notice that \eqref{isolated_alpha-} along with \ref{hyp_levelsets_bis} implies that
\begin{equation}\label{rhocT_bound}
\forall t \leq t_0^-, \hspace{2mm} \rho_{c,T}(t) \leq \frac{(r_0^-)^2}{2}.
\end{equation}
The fact that $\bV_{c,T} \in \CA((-\infty,t_0^-))$ implies that $\rho \in \CC^2((-\infty,t_0^-))$. We have that
\begin{equation}\label{rhocT'}
\forall t < t_0^-, \hspace{2mm}\rho_{c,T}'(t)= 2\langle \bV_{c,T}(t)-\bv^-,\bV_{c,T}'(t) \rangle_{\scrL}
\end{equation}
and
\begin{equation}\label{rhocT''}
\forall t < t_0^-, \hspace{2mm} \rho_{c,T}''(t) = 2\lVert \bV_{c,T}'(t) \rVert_{\scrL}^2+2 \langle \bV_{c,T}(t)-\bv^-,\bV_{c,T}''(t) \rangle_{\scrL}
\end{equation}
so that, using \eqref{isolated_equation}
\begin{align}
\forall t < t_0^-, \hspace{2mm}  \rho_{c,T}''(t)& =  2\lVert \bV_{c,T}'(t) \rVert_{\scrL}^2+2\langle D_{\scrL}\CE(\bV_{c,T}(t)), \bV_{c,T}(t)-\bv^- \rangle_{\scrL}\\&-2c\langle \bV_{c,T}(t)-\bv^-,\bV_{c,T}'(t) \rangle_{\scrL},
\end{align}
which, by \eqref{isolated_omega} and \eqref{rhocT'} gives the inequality
\begin{equation}\label{rhocT_equation}
\forall t < t_0^-, \hspace{2mm} \rho_{c,T}''(t)+c\rho'(t) \geq 0.
\end{equation}
Inequality \eqref{rhocT_equation} implies that $\rho_{c,T}$ does not possess any strict local maximum. Therefore, we can find $\tilde{t}_1 < t_0^-$ such that $\rho_{c,T}$ is monotone in $(-\infty,\tilde{t}_1)$. In combination with the bound \eqref{rhocT_bound}, monotony implies the existence of $l \in [0,(r_0^-)^2/2]$ such that
\begin{equation}\label{isolated_l}
\lim_{t \to -\infty} \rho_{c,T}(t)=l.
\end{equation}
In order to conclude the proof, we need to show that $l=0$. Assume by contradiction that $l \not = 0$. Decreasing the value of $\tilde{t}_1$ if necessary, we find $\tilde{l}>0$ such that for all $t \leq \tilde{t}_1$ it holds $\CE(\bV_{c,T}(t)) \geq \tilde{l}$. Therefore, using \eqref{local_monotonicity} in \ref{hyp_levelsets_bis}, we have that
\begin{equation}
\forall t \leq \tilde{t}_1, \exists \delta(t)>0, \forall \theta \in (1-\delta(t),1+\delta(t)), \hspace{2mm} \frac{d}{d\lambda}(\CE(\bv^-+(\bV_{c,T}(t)-\bv^-))) \geq \omega(\tilde{l}),
\end{equation}
with $\omega(\tilde{l})>0$. In particular, it holds
\begin{equation}\label{isolated_local_monotonicity_2}
\forall t \leq \tilde{t}_1, \hspace{2mm} \langle D_\scrL\CE(\bV_{c,T}(t)), \bV_{c,T}(t)-\bv^- \rangle_{\scrL} \geq \omega(\tilde{l}),
\end{equation}
which by \eqref{rhocT'} and \eqref{rhocT''} gives
\begin{equation}\label{rhocT_omega}
\forall t \leq \tilde{t}_1, \hspace{2mm} \rho_{c,T}''(t)+c\rho_{c,T}'(t) \geq 2\omega(\tilde{l}).
\end{equation}

Suppose first that $\rho_{c,T}$ is non increasing in $(-\infty,\tilde{t}_1)$, meaning that for all $t \leq \tilde{t}_1$, $\rho_{c,T}(\tilde{t}_1) \leq \rho_{c,T}(t)$ and $\rho_{c,T}'(\tilde{t}_1) \leq 0$. Moreover, integrating \eqref{rhocT_omega} in $(s,\tilde{t}_1)$ for $s < \tilde{t}_1$, we obtain
\begin{equation}\label{rhocT_equation_2}
\forall s < \tilde{t}_1, \hspace{2mm} \rho_{c,T}'(\tilde{t}_1)-\rho_{c,T}'(s)+c(\rho_{c,T}(\tilde{t}_1)-\rho_{c,T}(s)) \geq 2\omega(\tilde{l})(\tilde{t}_1-s).
\end{equation}
which becomes
\begin{equation}
\forall s < \tilde{t}_1, \hspace{2mm} 2\omega(\tilde{l})(s-\tilde{t}_1) \geq \rho_{c,T}'(s)
\end{equation}
meaning that $\lim_{s \to -\infty}\rho_{c,T}'(s) =-\infty$, which contradicts \eqref{isolated_l}. Therefore, $\rho_{c,T}$ must be non decreasing in $(-\infty,\tilde{t}_1]$, which means that for all $t \leq \tilde{t}_1$, $\rho_{c,T}'(t) \geq 0$. We now fix $s< t <\tilde{t}_1$ and integrate \eqref{rhocT_omega} in $[s,t]$. We obtain
\begin{equation}
\rho_{c,T}'(t)-\rho_{c,T}'(s)+c(\rho_{c,T}(t)-\rho_{c,T}(s)) \geq 2\omega(\tilde{l})(t-s),
\end{equation}
which, since $\rho_{c,T}$ is non decreasing and we have the bound \eqref{rhocT_bound}, gives
\begin{equation}
\rho_{c,T}'(t) \geq 2\omega(\tilde{l})(t-s)-c\frac{(r_0^-)^2}{2}
\end{equation}
so that, by integrating again
\begin{equation}
\forall s < \tilde{t}_2, \hspace{2mm} \rho_{c,T}(\tilde{t}_1)-\rho_{c,T}(s) \geq \omega(\tilde{t})(\tilde{t}_1-s)^2-c\frac{(r_0^-)^2}{2}(\tilde{t}_1-s)
\end{equation}
which gives a contradiction by passing to the limit $s \to -\infty$, since the right-hand side term tends in that case to $+\infty$. Therefore, we must have that the limit $l$ defined in \eqref{isolated_l} equals 0, which concludes the proof.
\end{proof}
\begin{remark}
The arguments in Lemma \ref{LEMMA_isolated} do not seem to apply to the case in which $\scrF^-$ is no longer a singleton. Indeed, in such an event, the function $\rho_{c,T}$ would have to be defined as
\begin{equation}
\rho_{c,T}: t \in (-\infty,t_0^-) \to \lVert \bV_{c,T}(t)-P^-(\bV_{c,T}(t)) \rVert_{\scrL}^2
\end{equation}
so that, when computing its derivatives, the (possibly nonzero) differential of $DP^-$ appears and with it some terms which do not seem to have necessarily the right sign. Recall also that the limit property given by Lemma \ref{LEMMA_isolated} does not allow, by itself, to obtain information on $\CE(\bV_{c,T}(\cdot))$ at $-\infty$, since $\CE$ is only lower semicontinuous with respect to $\scrL$-convergence.
\end{remark}
\subsection{Proof of Theorem  \ref{THEOREM_abstract_isolated} completed}
It is a particular case of Lemma \ref{LEMMA_isolated}. Indeed, the solution $(c^\star,\bU)$ given by Theorem \ref{THEOREM-ABSTRACT} is given by Proposition \ref{PROPOSITION_existence_bis}, so in particular is a constrained solution in $Y_T$ for suitable $T \geq 1$ and Lemma \ref{LEMMA_isolated} applies.

\qed
\subsection{Proof of Theorem \ref{THEOREM_abstract_bc} completed}
If  \ref{hyp_levelsets_bis} holds, then using Proposition \ref{PROPOSITION_existence_bis}, inequality \eqref{sup_D_bound} in Lemma \ref{LEMMA_set_CD} implies that $c^\star < \gamma^-$, where we have also used \eqref{a_convergence_bis} in \ref{hyp_convergence}. The proof then follows by using identity \eqref{solutions_limit-_norm} in Proposition \ref{PROPOSITION_conv_sol_-}.

\qed

\subsection{On the proof of Theorem \ref{THEOREM_abstract_speed}}

The proof works exactly as the proof of \textit{Theorem 7} in \cite{oliver-bonafoux-tw}, so we skip it.

\section{Proofs of the main results}
The main results of this paper are obtained by the following, which links the main setting with the abstract setting:
\begin{lemma}\label{LEMMA_implication_W}
Assume that \ref{asu_unbalancedW}, \ref{asu_projection_inverse} and \ref{asu_conv_mon} hold. Let
\begin{equation}
\scrL=\scrH=\tilde{\scrH}=\R^k
\end{equation}
and $\CE=W$. Then, \ref{hyp_unbalanced}, \ref{hyp_smaller_space}, \ref{hyp_compactness}, \ref{hyp_projections_inverse}, \ref{hyp_projections_H}, \ref{hyp_regularity} and \ref{hyp_levelsets_bis} hold. Moreover, if \ref{asu_bc} holds, then \ref{hyp_convergence} holds.
\end{lemma}
\begin{proof}
The proof is straightforward. Indeed, assumption \ref{hyp_unbalanced} follows from \ref{asu_unbalancedW}. Assumption \ref{hyp_smaller_space} becomes tautological because $\scrL=\scrH=\tilde{\scrH}=\R^k$. Assumption \ref{hyp_compactness} is straightforward. Assumption \ref{hyp_projections_inverse} corresponds directly to \ref{asu_projection_inverse}. Assumption \ref{hyp_projections_H} is also tautological. Regarding \ref{hyp_regularity}, the existence of the map $\FP$ follows from \eqref{W_R0}, see the proof of \textit{Lemma 2.3} in \cite{oliver-bonafoux-tw}. The rest of the statement follows from elliptic regularity results or it is tautological. Finally, we clearly have that \ref{hyp_convergence} holds if \ref{asu_bc} holds.
\end{proof}

Notice that most of the strength of the abstract results will not be used as we choose all the Hilbert spaces equal to $\R^k$. This also makes all the 
\subsection{Proof of Theorem  \ref{THEOREM_W_main} completed}
We apply Theorem \ref{THEOREM-ABSTRACT} with the choice of objects given by Lemma \ref{LEMMA_implication_W}. It is clear that \eqref{abstract_equation} reads now as \eqref{PROFILE_eq_Thm} and \eqref{abstract_bc_weak} reads as \eqref{conditions_infinity}, which establishes the existence part as we also have that $\CA(\R)$ becomes a subset of $\CC^2_{\loc}(\R,\R^k)$ by its definition in \eqref{CA_def}. Due to the fact that $X=\CS$, $Y=\overline{\CS}$ (with $X$ as in \eqref{X} and $Y$ as in \eqref{Y}) the uniqueness statement of Theorem \ref{THEOREM_W_main} follows directly from that in Theorem \ref{THEOREM-ABSTRACT}. Finally, it is clear that the exponential convergence in Theorem \ref{THEOREM-ABSTRACT} writes as that in Theorem \ref{THEOREM_W_main}, which concludes the proof.
\qed
\subsection{Proof of Theorem \ref{THEOREM_W_bc} completed}
We apply Theorems \ref{THEOREM_abstract_isolated} and \ref{THEOREM_abstract_bc} with the choice of objects given by Lemma \ref{LEMMA_implication_W}. Then, we have that the case 1. in \ref{asu_bc} corresponds to the assumptions of Theorem \ref{THEOREM_abstract_isolated} and case 2. in \ref{asu_bc} is the abstract assumption \ref{hyp_convergence}, required in Theorem \ref{THEOREM_abstract_bc}. Therefore, Theorem \ref{THEOREM_W_bc} follows.
\qed 
\subsection{Proof of Theorem \ref{THEOREM_W_speed} completed}
We apply Theorem \ref{THEOREM_abstract_speed} with the choice of objects given by Lemma \ref{LEMMA_implication_W}. It is then clear that Theorem \ref{THEOREM_W_speed} holds, as recall that $Y=\overline{\CS}$.
\qed
\printbibliography

\end{document}